\documentclass[reqno,10pt,centertags]{amsart}
\usepackage{amssymb}
\usepackage{hyperref}
\usepackage{latexsym}
\usepackage{amsmath}
\usepackage{mathrsfs}   
\usepackage[usenames,dvipsnames,svgnames,table]{xcolor}
\usepackage{orcidlink}

\allowdisplaybreaks

\newcommand*{\mailto}[1]{\href{mailto:#1}{\nolinkurl{#1}}}

   \def\sH{{\mathfrak H}}

      \def\dC{{\mathbb C}}

   \def\dN{{\mathbb N}}   
      \def\dR{{\mathbb R}}

\def\cA{{\mathcal A}}   \def\cB{{\mathcal B}}   
      
\def\cG{{\mathcal G}}   \def\cH{{\mathcal H}}   
      \def\cL{{\mathcal L}}

\def\R{\mathbb{R}}
\def\C{\mathbb{C}}

\def\ran{{\text{\rm ran\,}}}

\def\dom{{\text{\rm dom\,}}}

\def\min{{\rm min\,}}
\def\max{{\rm max\,}}

\DeclareMathOperator{\Real}{Re}

\DeclareMathOperator*{\dist}{dist}

\newtheorem{theorem}{Theorem}[section]
\newtheorem*{thm*}{Theorem}
\newtheorem{proposition}[theorem]{Proposition}
\newtheorem{corollary}[theorem]{Corollary}
\newtheorem{lemma}[theorem]{Lemma}

\theoremstyle{definition}
\newtheorem{definition}[theorem]{Definition}
\newtheorem{example}[theorem]{Example}

\newtheorem{assumption}[theorem]{Assumption}

\newtheorem{remark}[theorem]{Remark}

\numberwithin{equation}{section}

\begin{document}
\title[Spectral Properties of 1D Fractional Laplacians]{On Spectral Properties of Restricted Fractional Laplacians with Self-adjoint Boundary Conditions on a Finite Interval}

\author[J.\ Behrndt]{Jussi Behrndt \orcidlink{0000-0002-3442-6777}}
\address{Jussi Behrndt, Institut f\"ur Angewandte Mathematik, Technische Universit\"at
Graz,\\ Steyrergasse 30, 8010 Graz, Austria}
\email{behrndt@tugraz.at}
\urladdr{www.math.tugraz.at/$\sim$behrndt/}

\author[M.\ Holzmann]{Markus Holzmann \orcidlink{0000-0001-8071-481X)}}
\address{Markus Holzmann, Institut f\"ur Angewandte Mathematik, Technische Universit\"at
Graz,\\ Steyrergasse 30, 8010 Graz, Austria}
\email{holzmann@math.tugraz.at}
\urladdr{www.math.tugraz.at/$\sim$holzmann/}

\author[D.\ Mugnolo]{Delio Mugnolo \orcidlink{0000-0001-9405-0874}}
\address{Delio Mugnolo, Lehrgebiet Analysis, Fakultät Mathematik und Informatik, Fern\-Universität in Hagen, D-58084 Hagen, Germany}
\email{delio.mugnolo@fernuni-hagen.de}
\urladdr{https://www.fernuni-hagen.de/analysis/team/delio.mugnolo.shtml}



\begin{abstract}
We describe all self-adjoint realizations of the restricted fractional Laplacian $(-\Delta)^a$ with power $a \in (\frac{1}{2}, 1)$ on a bounded interval by imposing boundary conditions on the functions in the domain of  
a maximal realization; such conditions relate suitable weighted Dirichlet and Neumann traces. This is done  in a systematic way by using the abstract
concept of boundary triplets and their Weyl functions from extension and spectral theory of symmetric and self-adjoint operators in Hilbert spaces.
Our treatment
follows closely the well-known one for classical Laplacians on intervals and it shows that all self-adjoint realizations have purely discrete spectrum and are semibounded from below. To demonstrate the method, we focus on three self-adjoint realizations of the restricted fractional Laplacian: the Friedrichs extension, corresponding to Dirichlet-type boundary conditions, the Krein--von Neumann extension, and a Neumann-type realization. Notably, the 
Neumann-type realization exhibits a simple negative eigenvalue, thus it is not larger than the Krein--von Neumann extension.
\end{abstract}

\maketitle

\section{Introduction} 

Fractional Laplace operators, which were introduced already in~\cite{Boc49, Rie39}, appear nowadays in several fields, including relativistic \cite{MazVer04} and non-relativistic \cite{BogByc99} quantum mechanics (see  \cite{BogMer24, FraLieSei08} and the references therein), anomalous diffusion~\cite{MeeSik19}, stochastic analysis~\cite{SonVon03}, turbulence~\cite{Bak08}, and optics~\cite{BarBerWie08}. For physical systems supported on the whole space the associated operator is extensively studied, both in modeling and analysis. Owing to the spectral theorem and the subordination theory \cite{MarSan01}, in this case there is no ambiguity when it comes to defining fractional operators: no less than ten different
but equivalent definitions of the fractional Laplacian on $\R^d$ are known \cite{Kwa17}. Our aim in this paper is to study the less clear case of bounded domains from an extension theoretic point of view. More precisely, we consider a specific variant of the fractional Laplacian, which we call, with a slight abuse of terminology, \textit{restricted} fractional Laplacian: we are going to parametrize all its self-adjoint realizations by imposing suitable boundary conditions. This natural question is unanswered even in the one-dimensional case of intervals, which is considered in the present paper.

Throughout this paper, let $a\in (\frac{1}{2},1)$ be fixed. The fractional Laplacian $(-\Delta)^a$ on $\mathbb{R}$ is the strongly singular integral operator that acts on any $f$ belonging to the Sobolev space $H^{2a}(\mathbb{R})$ of fractional order $2a$ as
\begin{equation} \label{def_fract_Laplace_R}
  (-\Delta)^a f(x) = c_{a} \lim_{\varepsilon \searrow 0} \int_{|x-y| > \varepsilon} \frac{f(x)-f(y)}{|x-y|^{1+2a}} d y, \quad x \in \mathbb{R},
\end{equation}
where $c_a>0$ is the constant in~\eqref{def_c_a}. 
Equivalently, this operator can be described with the Fourier transform $\mathcal{F}$ in $L^2(\mathbb{R})$ via $\mathcal{F} ((-\Delta)^a f)(\xi) = \xi^{2a} \mathcal{F} f(\xi)$, or as the fractional power of the free Laplacian defined on $H^2(\mathbb{R})$ via the spectral theorem for self-adjoint operators. In particular, the representation via the Fourier transform shows that $(-\Delta)^a$ can be viewed as a bounded operator from $H^{2a}(\mathbb{R})$ to $L^2(\mathbb{R})$ and thus, by duality and self-adjointness, it admits a bounded extension mapping from $L^2(\mathbb{R})$ to $H^{-2a}(\mathbb{R})$.

The  main objective of this paper is to parametrize and to study the spectral properties of the self-adjoint realizations of the 
fractional Laplacian $(-\Delta)^{a}$ on a bounded interval $(\alpha,\beta)\subset\mathbb R$ by using boundary triplets and 
Weyl functions -- an efficient tool in abstract extension theory of symmetric operators and spectral analysis
of their self-adjoint extensions; cf. Section~\ref{sec2}. We emphasize that on a bounded interval there are different, not mutually equivalent versions of fractional Laplacians. Inspired by the work of Grubb (see, e.g., \cite[Section~6.1]{G16}),
we will work with a variant of $(-\Delta)^{a}$, which we call \textit{restricted fractional Laplacian}, and which acts on a function $f \in L^2(\alpha, \beta)$ by first extending it by zero to an element $\mathfrak{e} f \in L^2(\mathbb{R})$, then by applying the fractional Laplacian $(-\Delta)^a$ on $\mathbb{R}$ to $\mathfrak{e} f$, and finally by restricting $(-\Delta)^a \mathfrak{e} f \in H^{-2a}(\mathbb{R})$ to a distribution $\mathfrak{r} (-\Delta)^a \mathfrak{e} f \in \mathcal{D}'(\alpha, \beta)$ on $(\alpha, \beta)$. The maximal realization of the restricted fractional Laplacian in $L^2(\alpha, \beta)$ is then defined for all those $f \in L^2(\alpha, \beta)$ such that $\mathfrak{r} (-\Delta)^a \mathfrak{e} f \in L^2(\alpha, \beta)$. Note that by construction, the action of the restricted fractional Laplacian can also be described as the strongly singular integral in~\eqref{def_fract_Laplace_R}.
Let us briefly remind the reader about other common types of fractional Laplacians on domains.
The simplest one -- the \textit{spectral fractional Laplacian} -- is also the least appropriate for our purposes, as it is defined by taking fractional powers (via the spectral theorem) of a non-negative realization of the Laplacian; in this way, boundary conditions are implicitly already prescribed, see \cite{G16} and the references therein. Another way of defining fractional powers that is tailored for a Neumann-like realization with nonlocal boundary conditions and has received much attention in the last years is described in~\cite{DROV17}; again, as it is related to fixed boundary conditions, it is not suitable for the purposes of extension theory.
Finally, the so-called \textit{regional fractional Laplacian} arises as the $L^2$-gradient of the Slobodeckii-seminorm. This realization is natural from the viewpoint of stochastic analysis, as it generates censored stable processes. Possibly for this reason, its Dirichlet, Neumann, and Robin-type extensions have been  studied (see e.g.~\cite{BBC03,GS19,G06,GM06,W15}), but to the best of our knowledge no systematic characterization of all self-adjoint realizations and their spectral properties has been obtained so far. 
We emphasize that the  Dirichlet realizations of the restricted and the regional fractional Laplacian only differ by a positive potential; this was observed already in~\cite[Section~2]{BBC03}, see also \cite[Section~6.3]{G16} and \cite{GS19,W15}.

To make this paper easily readable we first familiarize the reader with the general abstract notions 
of boundary triplets and Weyl functions in Section~\ref{sec2}, and 
illustrate in parallel these concepts for the differential expression $-\Delta$ on a finite interval $(\alpha,\beta)$, which is 
a particularly simple example and well known to operator theory experts. However, we believe that this 
special case already displays many advantages of boundary triplets and Weyl functions and, at the same time, provides 
an easy access to the abstract theory. In particular, using the Dirichlet and Neumann trace operators we parametrize 
the self-adjoint realizations of $-\Delta$ in $L^2(\alpha,\beta)$ with 
the help of self-adjoint relations $\Theta\subset \dC^2\times\dC^2$ or pairs of $2\times 2$-matrices $\cA, \cB$ 
satisfying a symmetry and a maximality condition. The Weyl function in this case  
is the ($2\times 2$-matrix valued) Dirichlet-to-Neumann map, which contains the spectral data of the self-adjoint realizations.
Furthermore, we briefly discuss semibounded self-adjoint extensions and special attention is paid to 
the Friedrichs and Krein--von Neumann extension (in the nonnegative case).
It is noteworthy that the boundary conditions of the Friedrichs extension can
be identified via the limiting behavior of the Weyl function at $-\infty$; cf. \cite{BHS20,DM91,DM95} and Proposition~\ref{sf} .

After the preparations and motivations in Section~\ref{sec2} we turn to the main objective of this paper.  
To fully describe the domains of the self-adjoint realizations of the restricted fractional Laplacian
$(-\Delta)^{a}$ in $L^2(\alpha,\beta)$ suitable function spaces are necessary. While for $a=1$ the $L^2$-based Sobolev space $W^{2,2}(\alpha,\beta)$ of twice weakly differentiable functions
is the natural choice, for the restricted fractional Laplacian with $a<1$ the so-called H\"ormander transmission spaces turn out to be useful; cf. Section~\ref{sec3}. These spaces, originally introduced by H\"ormander in \cite{H65}, played an important role in the recent publications 
\cite{G14, G15, G16, G18-1, G18} by Grubb, where properties of solutions of equations involving the restricted fractional Laplacian were studied. The works by
Grubb also serve as an inspiration and source for our studies. 
The definition of the H\"ormander transmission spaces $H^{\mu(s)}([\alpha, \beta])$ and  some of their properties that are useful 
in the present context are summarized in Section~\ref{section_Hoermander_transmission_spaces}.
For $a \in (\frac{1}{2}, 1)$ the minimal restricted fractional Laplacian $(-\Delta)_\min^{a}$, the maximal restricted fractional Laplacian $(-\Delta)_\max^{a}$, and the 
Dirichlet realization $(-\Delta)^a_{\textup{D}}$ of the restricted fractional Laplacian 
are then defined in Section~\ref{section_minimal_maximal_operator}. It is noteworthy that for functions in the domain of $(-\Delta)^a_\max$ weighted Dirichlet and Neumann traces that appear in the boundary conditions~\eqref{boundary_conditions_intro} below can be defined. Regularity of solutions of elliptic problems involving $(-\Delta)^a_\textup{D}$ has been studied in \cite{G14, G15, ROS14-1, ROS14-2}. Moreover, while the spectrum of $(-\Delta)^a_\textup{D}$ is not easy to describe, it has already been studied in detail, see, e.g., the survey~\cite{SV14} and in particular~\cite[Theorem~1.1]{SV14}, where the ground state energy of the restricted fractional Dirichlet Laplacian on a bounded interval is compared with that of the spectral fractional Dirichlet Laplacian (see also \cite{CheSon05}). Simplicity of all eigenvalues was shown in~\cite{K12}, while spectral bounds and Weyl-type asymptotics were obtained in~\cite{BluGet59, D12}.

Based on the previous considerations 
by Grubb and a Green's identity from \cite{G18-1}  we then obtain a boundary triplet for 
the maximal restricted fractional Laplacian $(-\Delta)_\max^{a}$ in Section~\ref{btsec}.
This leads directly to a description of all self-adjoint fractional Laplacians as the restriction of $(-\Delta)^a_\max$ to those $f$ that satisfy the boundary conditions
\begin{equation} \label{boundary_conditions_intro}
  \Gamma(a) \mathcal{B}  \begin{pmatrix} (t^{1-a}f)(\alpha) \\ (t^{1-a} f)(\beta) \end{pmatrix}
  = \Gamma(a+1) \mathcal{A}  \begin{pmatrix} (t^{1-a}f)'(\alpha) \\ -(t^{1-a} f)'(\beta) \end{pmatrix}, 
\end{equation}
where $\Gamma$ is Euler's gamma function, $t(x) = \min \{x-\alpha, \beta-x\}$, and $\mathcal{A}, \mathcal{B}\in\mathbb{C}^{2 \times 2}$ are such that $\cA\cB^*= \cB\cA^*$
and $(\cA \, | \, \cB) \in \mathbb{C}^{2 \times 4}$ has rank  2; cf. Proposition~\ref{proposition_self_adjoint}. This characterization is also the reason for the restriction $a \in (\frac{1}{2}, 1)$, as for $a \leq \frac{1}{2}$ the weighted Neumann trace of functions in the domain of $(-\Delta)^a_\max$ can not be defined, and for $a > 1$ it is expected that the traces appearing in~\eqref{boundary_conditions_intro} are not sufficient to describe \textit{all} self-adjoint realizations, as for this purpose, e.g., also boundary conditions involving the evaluation of the weighted second derivative on the boundary is necessary. Next, using
results from \cite{D12} we compute the 
corresponding Weyl function $M(\lambda)$ at $\lambda=0$ and relying on abstract properties of the Weyl function we obtain 
the form of $M(\lambda)$ for all $\lambda$, and also identify the Friedrichs extension via the 
limiting behavior of $M$ for $\lambda\rightarrow -\infty$; here we use typical boundary triplet techniques from \cite{BHS20,DM91,DM95} and also rely on a result from \cite{BLLR18}.
Furthermore, we describe the Krein--von Neumann realization of the restricted fractional Laplacian, as well as the 
nonnegative and semibounded self-adjoint extensions. An interesting feature appears for the Neumann realization: In contrast to the 
classical case $a=1$, the Neumann realization of the restricted fractional Laplacian for $a \in (\frac{1}{2}, 1)$
possesses a simple negative eigenvalue. In particular, it is \textit{not larger} than the Krein--von Neumann realization and, unlike for $a=1$, it is \textit{not} associated with a (sub-)Markovian stochastic process. 
We wish to point out that formally 
the observations in Section~\ref{sec2} for the (ordinary) Laplacian $-\Delta$ can be viewed as limiting results of the findings in Section~\ref{section_fractional} for $a\rightarrow 1$.

\subsection*{Notations}

For a linear operator $A$ from a Hilbert space $\cG$ to $\cH$ the domain, range, and kernel are denoted by $\dom A, \ran A$, and $\ker A$, respectively. The space of all bounded linear operators from $\cG$ to $\cH$ is $\cL(\cG, \cH)$; if $\cG = \cH$, then we simply write $\cL(\cG) := \cL(\cG, \cG)$. If $A$ is a closed operator in $\cH$, then its resolvent set, spectrum, point spectrum, and continuous spectrum are denoted by $\rho(A), \sigma(A)$, $\sigma_\textup{p}(A)$, and $\sigma_\textup{c}(A)$, respectively. For a self-adjoint operator $A$ we use the symbol $\sigma_\textup{d}(A)$ for the discrete and $\sigma_\textup{ess}(A)$ for the essential spectrum.

\section{Boundary triplets and Weyl functions} \label{sec2}

The aim of this section is to recall the notion of boundary triplets and Weyl functions from extension theory of symmetric and self-adjoint operators.
We refer the reader to the textbook \cite{BHS20} and the review paper \cite{BGP08} for a comprehensive treatment of the subject; to \cite{B76,DM91,DM95,K75} for the origins of these concepts and to the textbooks \cite{G09,GG91,S12}.
To make this paper easily readable and to prepare the reader for the fractional order case in Sections~\ref{sec3} and \ref{section_fractional} we 
illustrate the abstract concepts in parallel for $-\Delta$ on a finite interval $(\alpha,\beta)$ in Examples~\ref{superklar}, \ref{exi}, \ref{example_KvN},
and \ref{example_ordering}.

\subsection{Boundary triplets and self-adjoint extensions}
Throughout this section $S$ is assumed to be a closed symmetric operator in a separable Hilbert space $\sH$ with dense domain $\dom S$.
It follows that the adjoint operator $S^*$ is well defined, closed, and an extension of $S$, that is, $S\subset S^*$.

\begin{definition}\label{yes}
A {\it boundary triplet} $\{\cG,\Upsilon_0,\Upsilon_1\}$ for  $S^*$ consists of a Hilbert space $\cG$ 
and two linear mappings $\Upsilon_0,\Upsilon_1:\dom S^*\rightarrow\cG$ such that the following properties hold:
\begin{itemize}
 \item [{\rm (i)}] the abstract Green identity 
\begin{equation*}
 (S^*f,g)-(f,S^*g)=(\Upsilon_1 f,\Upsilon_0 g)-(\Upsilon_0 f,\Upsilon_1 g)
\end{equation*}
is valid for all $f,g\in\dom S^*$,
 \item [{\rm (ii)}] the mapping $(\Upsilon_0,\Upsilon_1)^\top:\dom S^*\rightarrow\cG\times\cG$ is onto.
\end{itemize}
\end{definition}

A boundary triplet for $S^*$ exists if and only if the defect numbers $n_\pm(S)=\dim\ker(S^*\mp i)$ of $S$ coincide, and in this case one has 
$n_\pm(S)=\dim\cG\leq \infty$.
Note also that a boundary triplet is not unique (unless $n_\pm(S)=0$). In the following let us assume that $\{\cG,\Upsilon_0,\Upsilon_1\}$ is a boundary triplet 
for $S^*$. Then it follows that 
\begin{equation}\label{jaja}
A_0:=S^*\upharpoonright\ker\Upsilon_0\quad\text{and}\quad A_1:=S^*\upharpoonright\ker\Upsilon_1
\end{equation}
are self-adjoint extensions of $S$ and one has $\dom S=\ker\Upsilon_0\cap\ker\Upsilon_1$. Moreover, if $\dom S^*$ is equipped with the graph norm, 
then the boundary mapping in Definition~\ref{yes}~(ii) is continuous, and $\Upsilon_0,\Upsilon_1:\dom S^*\rightarrow\cG$ are surjective. 
With the help of the boundary triplet $\{\cG,\Upsilon_0,\Upsilon_1\}$ one can parametrize all extensions $A_\Theta\subset S^*$ of $S$ via linear relations 
$\Theta$ in $\cG$ (i.e. via the closed linear subspaces of $\cG \times \cG$). More precisely, there is a bijective correspondence  $\Theta\mapsto A_\Theta$ between the set of intermediate extensions $A_\Theta$ of $S$ and the set of linear 
relations in $\cG$  via
\begin{equation}\label{atheta}
 A_\Theta f =S^* f,\quad \dom A_\Theta=\left\{f\in\dom S^*: \begin{pmatrix} \Upsilon_0 f \\ \Upsilon_1 f\end{pmatrix}\in\Theta\right\}.
\end{equation}
One has $\overline{A_\Theta}=A_{\overline \Theta}$ and $(A_\Theta)^*=A_{\Theta^*}$ and, in particular, $A_\Theta$ is self-adjoint (symmetric, essentially self-adjoint) in $\sH$
if and only if $\Theta$ is self-adjoint (symmetric, essentially self-adjoint, respectively) in $\cG$; similar statements hold for (maximal) dissipative and accretive extensions; cf. \cite[Chapter 2.1]{BHS20}.
Observe that in the special case of a linear operator $\Theta$ in \eqref{atheta} one has 
\begin{equation}\label{atheta2}
\begin{split}
 A_\Theta f &=S^* f,\\ \dom A_\Theta&=\bigl\{f\in\dom S^*: \Theta\Upsilon_0 f=\Upsilon_1 f\bigr\}=\ker(\Upsilon_1-\Theta\Upsilon_0),
\end{split}
 \end{equation}
(and this characterization remains valid also in the case of a general linear relation $\Theta$ if the expression $\Upsilon_1-\Theta\Upsilon_0$ is understood in the sense of linear relations).
Note that $A_0$ and $A_1$ in \eqref{jaja} correspond to the linear relations 
\begin{equation*}
\left\{\begin{pmatrix} 0 \\ g'\end{pmatrix}:g'\in\cG\right\}\quad\text{and}\quad  \left\{\begin{pmatrix} g \\ 0\end{pmatrix}:g\in\cG\right\},
\end{equation*}
respectively.

In many situations one prefers abstract boundary conditions of the form 
\eqref{atheta2}, where $\Theta$ is an operator. However, for a complete 
description of
the (self-adjoint) extensions it is necessary to allow also multivalued 
parameters, or to use slightly more general forms of boundary conditions 
than in \eqref{atheta2}.
Making use of a representation of a closed (self-adjoint) linear 
relation $\Theta$ as a kernel of a row operator one can rewrite 
\eqref{atheta} using two bounded operators.
In the next proposition this is made explicit for relations $\Theta$ 
with $\rho(\Theta) \neq \emptyset$; cf. \cite[Proposition~1.10.4 and 
Corollary~1.10.5]{BHS20} for the existence of such a representation and 
\cite[Proposition~1.5]{BGP08} for the condition on the self-adjointness.

\begin{proposition}\label{abcd}
Let $\mu\in\C$, and let $\{\cG,\Upsilon_0,\Upsilon_1\}$ be a boundary 
triplet for $S^*$. Then $\Theta$ is a linear relation in $\cG$ with $\mu 
\in \rho(\Theta)$ if and only if $\Theta$ admits the representation
\begin{equation}\label{nanu33}
  \Theta=\left\{\begin{pmatrix} k \\ k'\end{pmatrix}:\cB k=\cA 
k'\,\,\text{for all}\,\, k,k'\in\cG \right\}
\end{equation}
with $\mathcal A,\mathcal B\in\mathcal L(\mathcal G)$ and $\cB - \mu 
\cA$ is bijective.
With this choice the adjoint of $\Theta$ is
\begin{equation}\label{nanu}
  \Theta^*=\left\{\begin{pmatrix} \mathcal A^* h \\ \mathcal B^* 
h\end{pmatrix}: h\in\cG \right\}
\end{equation}
and
the closed intermediate extension $A_{\Theta}$ in \eqref{atheta}
can be described as
\begin{equation}\label{nochamal}
\begin{split}
A_\Theta f &=S^* f,\\
\dom A_\Theta&=\left\{f\in\dom S^*: \mathcal 
B\Upsilon_0 f=\mathcal A\Upsilon_1 f\right\} = \ker (\cA \Upsilon_1 - 
\cB \Upsilon_0).
\end{split}
\end{equation}
Furthermore, with $\Theta$ in \eqref{nanu33} it follows that 
$A_{\Theta}$ in \eqref{nochamal} is self-adjoint in $\sH$
if and only if
\begin{equation}\label{nanu2}
\cA \cB^*= \cB \cA^*
\quad \mbox{and} \quad \ker \begin{pmatrix} \cB & -\cA \\ \cA & \cB 
\end{pmatrix} = \{ 0 \}.
\end{equation}
\end{proposition}

Note that for finite dimensional spaces $\mathcal{G} \cong 
\mathbb{C}^N$ the condition~\eqref{nanu2} about the self-adjointness of 
$\Theta$ in~\eqref{nanu33} is equivalent to
\begin{equation*}
\cA \cB^*= \cB \cA^*
\quad \mbox{and} \quad (\cA \, | \, \cB) \in \mathbb{C}^{N \times 2 N} 
\text{ has full rank};
\end{equation*}
cf. \cite[Corollary~1.6]{BGP08}.

The following elementary example illustrates how boundary triplets naturally appear in the context of ordinary differential operators.
The next example can also be viewed as the limiting case $a\rightarrow 1$ in our discussion of fractional operators in Section~\ref{section_fractional}.

\begin{example}\label{superklar}
In $\sH=L^2(\alpha,\beta)$, where $-\infty<\alpha<\beta<\infty$, we consider the minimal operator associated to $-\Delta$ given by 
\begin{equation*}
\begin{split}
 S f&=-f'',\\\dom S&=W^{2,2}_0(\alpha,\beta)=\bigl\{f\in W^{2,2} (\alpha,\beta): f(\alpha)=f'(\alpha)=f(\beta)=f'(\beta)=0\bigr\},
\end{split}
 \end{equation*}
which is a densely defined closed symmetric operator in $L^2(\alpha,\beta)$ with equal defect numbers $n_\pm(S)=2$; here and in the following we denote by $W^{2,2} (\alpha,\beta)$ the $L^2$-based Sobolev space
 of twice weakly differentiable functions and by $W^{2,2}_0(\alpha,\beta)$ the closure of $C_0^\infty(\alpha, \beta)$ in $W^{2,2}(\alpha, \beta)$.
The adjoint of $S$, the maximal operator, is given by 
\begin{equation*}
 S^* f=-f'',\quad \dom S^*=W^{2,2}(\alpha,\beta),
\end{equation*}
and it is easy to see that $\{\dC^2,\Upsilon_0,\Upsilon_1\}$, where 
\begin{equation*}
 \Upsilon_0 f=\begin{pmatrix} f(\alpha)\\ f(\beta)\end{pmatrix}\quad\text{and}\quad  \Upsilon_1 f=\begin{pmatrix} f'(\alpha)\\ -f'(\beta)\end{pmatrix},\quad f\in\dom S^*,
\end{equation*}
is a boundary triplet for $S^*$. In fact, the abstract Green identity in Definition~\ref{yes}~(i) follows from integration by parts and 
the surjectivity condition Definition~\ref{yes}~(ii) can be verified by choosing, e.g. suitable polynomials with prescribed boundary values at the finite endpoints 
$\alpha$ and $\beta$.
In the present situation we have 
\begin{equation}\label{a1a1}
\begin{split}
 A_0 f=-f'',&\quad \dom A_0=\bigl\{f\in W^{2,2} (\alpha,\beta): f(\alpha)=f(\beta)=0\bigr\},\\
 A_1 f=-f'',&\quad \dom A_1=\bigl\{f\in W^{2,2} (\alpha,\beta): f'(\alpha)=f'(\beta)=0\bigr\},
\end{split}
\end{equation}
so that the self-adjoint extensions $A_0$ and $A_1$ correspond to Dirichlet and Neumann boundary conditions, respectively.
For symmetric matrices $\Theta\in \dC^{2\times 2}$ the self-adjoint extension $A_\Theta$ in \eqref{atheta}-\eqref{atheta2} is given by
\begin{equation*}
 A_\Theta f=-f'',\quad \dom A_\Theta=\left\{f\in W^{2,2}(\alpha,\beta): \Theta\begin{pmatrix} f(\alpha)\\ f(\beta)\end{pmatrix}= \begin{pmatrix} f'(\alpha)\\ -f'(\beta)\end{pmatrix}\right\}
\end{equation*}
(here $\cA=I$ and $\cB=\Theta$ in \eqref{nanu33}) and, more generally, for any pair of matrices $\cA,\cB\in \dC^{2\times 2}$ such that 
$\cA \cB^*= \cB \cA^*$ and $(\mathcal A \,|\, \mathcal B)$ has rank two the corresponding self-adjoint extension $A_\Theta$ with $\Theta$ as in \eqref{nanu33} is given by
\begin{equation*}
 A_\Theta f=-f'',\quad \dom A_\Theta=\left\{f\in W^{2,2}(\alpha,\beta): \cB\begin{pmatrix} f(\alpha)\\ f(\beta)\end{pmatrix}=\cA \begin{pmatrix} f'(\alpha)\\ -f'(\beta)\end{pmatrix}\right\}.
\end{equation*}
\end{example}

In the end of this subsection we recall a theorem  of a certain inverse nature, which can be used to prove that a given 
operator $T$ is the adjoint of a symmetric operator $S$. At the same time one then obtains a boundary triplet for $T=S^*$; cf. \cite[Theorem 2.1.9]{BHS20}.

\begin{theorem}\label{inv}
Let $T$ be an operator in $\sH$, let $\cG$ be a Hilbert space, and assume that 
$\Upsilon_0, \Upsilon_1:\dom T\rightarrow\cG$
are linear mappings such that the following conditions hold:
\begin{itemize}
\item [{\rm (i)}] the abstract Green identity
\[
(Tf,g)-(f,Tg)
=(\Upsilon_1 f,\Upsilon_0 g)-(\Upsilon_0 f,\Upsilon_1 g)
\] 
is valid for all $f,g\in \dom T$,
\item [{\rm (ii)}] $(\Upsilon_0,\Upsilon_1)^\top:\dom T\rightarrow\cG\times\cG$ is onto and $\ker\Upsilon_0\cap\ker\Upsilon_1$ is dense in $\sH$,
\item [{\rm (iii)}] $T\upharpoonright\ker\Upsilon_0$ contains a self-adjoint operator $A_0$.
\end{itemize}
Then 
$$S:=T\upharpoonright\bigl(\ker\Upsilon_0\cap\ker\Upsilon_1\bigr)$$ 
is a densely defined closed
symmetric operator in $\sH$ such that $S^*=T$ and
$\{\cG,\Upsilon_0,\Upsilon_1\}$ is a boundary triplet for $S^*$ with
$A_0=S^*\upharpoonright\ker\Upsilon_0$.
\end{theorem}

\subsection{The Weyl function and Krein's formula for self-adjoint extensions}
\label{section_Weyl_function_abstract}

Next we shall recall the definitions and some properties of the $\gamma$-field  and Weyl function corresponding to a boundary triplet; we refer to \cite[Chapter 2.3]{BHS20} for details and full proofs. We make use of the direct sum domain decomposition 
\begin{equation*}
 \dom S^*=\dom A_0  \dot + \ker(S^*-\lambda)=\ker \Upsilon_0  \dot + \ker(S^*-\lambda),\quad\lambda\in\rho(A_0),
\end{equation*}
and from this it is clear that $\Upsilon_0\upharpoonright \ker(S^*-\lambda)$, $\lambda\in\rho(A_0)$, is invertible.

\begin{definition} \label{def_gamma_Weyl_abstract}
Let $\{\cG,\Upsilon_0,\Upsilon_1\}$ be a boundary triplet for $S^*$ and let $A_0=S^*\upharpoonright\ker\Upsilon_0$. Then the corresponding 
$\gamma$-field is defined by
\begin{equation*}
 \rho(A_0)\ni\lambda\mapsto \gamma(\lambda)=\bigl(\Upsilon_0\upharpoonright \ker(S^*-\lambda)\bigr)^{-1}
\end{equation*}
and the corresponding Weyl function is defined by 
\begin{equation*}
 \rho(A_0)\ni\lambda\mapsto M(\lambda)=\Upsilon_1\bigl(\Upsilon_0\upharpoonright \ker(S^*-\lambda)\bigr)^{-1}.
\end{equation*}
\end{definition}

One verifies that $\gamma(\lambda)\in\cL(\cG,\sH)$ and $M(\lambda)\in\cL(\cG)$ for all $\lambda\in\rho(A_0)$. For $\lambda,\mu\in\rho(A_0)$ one has the identities 
\begin{equation}\label{gamm2}
 \gamma(\lambda)=\bigl(1+(\lambda-\mu)(A_0- \lambda)^{-1}\bigr)\gamma(\mu)\quad\text{and}\quad \gamma(\lambda)^*=\Upsilon_1(A_0-\overline\lambda)^{-1},
\end{equation}
which are often useful. In particular, it follows that $\gamma$ is analytic on $\rho(A_0)$.
Moreover, 
$M(\lambda)=\Upsilon_1\gamma(\lambda)$, $\lambda\in\rho(A_0)$,  and the function $\lambda\mapsto M(\lambda)$ is a (operator-valued) Nevanlinna (or Riesz--Herglotz) 
function on $\rho(A_0)$. In particular, for $\eta\in\rho(A_0)\cap\dR$ one has $M(\eta)=M(\eta)^*$ and $M$ is monotonously increasing on 
intervals $I\subset\rho(A_0)\cap\dR$. We also note that for $\lambda,\mu\in\rho(A_0)$ the identities 
\begin{equation*}
 M(\lambda)-M(\mu)^*=(\lambda-\overline\mu)\gamma(\mu)^*\gamma(\lambda)
\end{equation*}
and
\begin{equation} \label{Weyl2}
 M(\lambda)=\Real M(\mu)+\gamma(\mu)^*\bigl[(\lambda-\Real\mu)+(\lambda-\mu)(\lambda-\overline\mu)(A_0-\lambda)^{-1}\bigr]\gamma(\mu) 
\end{equation}
are valid.

We will now return to Example~\ref{superklar} and provide the $\gamma$-field and Weyl function.

\begin{example} \label{exi}
In order to compute the $\gamma$-field and Weyl function corresponding to the boundary triplet in Example~\ref{superklar} we note first that for $\lambda\not=0$ 
\begin{equation*}
 \ker(S^* -\lambda)=\text{span}\,\bigl\{x\mapsto \cos\bigl(\sqrt{\lambda}(x-\alpha)\bigr),x\mapsto \sin\bigl(\sqrt{\lambda}(x-\alpha)\bigr)\bigr\}
\end{equation*}
and that the spectrum of $A_0$ consists of the eigenvalues $\lambda_k=(k\pi)^2/(\beta-\alpha)^2$, $k\in\dN$. For $\lambda\in\rho(A_0)$, $\lambda\not=0$, the $\gamma$-field $\gamma(\lambda):\dC^2\rightarrow L^2(\alpha,\beta)$ 
has the following explicit form
\begin{equation*}
 \gamma(\lambda)\begin{pmatrix}c_1 \\ c_2\end{pmatrix} =c_1 \cos\bigl(\sqrt{\lambda}(\dot x-\alpha)\bigr) +\frac{c_2-c_1 \cos\bigl(\sqrt{\lambda}(\beta-\alpha)\bigr)}{\sin\bigl(\sqrt{\lambda}(\beta-\alpha)\bigr)} \sin\bigl(\sqrt{\lambda}(\dot x-\alpha)\bigr),\qquad \dot{x}\in (\alpha,\beta),
\end{equation*}
and the Weyl function is a $2\times 2$-matrix function of the form
\begin{equation}\label{mmm}
 M(\lambda)=\frac{\sqrt{\lambda}}{\sin\bigl(\sqrt{\lambda}(\beta-\alpha)\bigr)}
\begin{pmatrix} 
-\cos\bigl(\sqrt{\lambda}(\beta-\alpha)\bigr) & 1\\ 1 & -\cos\bigl(\sqrt{\lambda}(\beta-\alpha)\bigr)    
\end{pmatrix}.
\end{equation}
For $\lambda=0$ one obtains
\begin{equation}\label{g0}
 \gamma(0)\begin{pmatrix}c_1 \\ c_2\end{pmatrix} =c_1 +\frac{c_2-c_1}{\beta-\alpha}(\dot x-\alpha)
\end{equation}
and 
\begin{equation}\label{m0}
 M(0)=\frac{1}{\beta-\alpha}
\begin{pmatrix} 
- 1 & 1\\ 1 & - 1
\end{pmatrix}.
\end{equation}
\end{example}

For completeness we now also recall Krein's formula for self-adjoint extensions and refer to \cite[Chapter 2.6]{BHS20} for a more general and complete treatment 
of the spectral analysis of the extensions $A_\Theta$.

\begin{theorem}\label{Kreinformula}
Let $\{\cG,\Upsilon_0,\Upsilon_1\}$
be a boundary triplet for $S^*$, $A_0=S^*\upharpoonright\ker\Upsilon_0$, and let $\gamma$ and $M$ be
the corresponding $\gamma$-field and Weyl function, respectively. Moreover, let
$\Theta$ be a self-adjoint operator or relation in $\cG$  and let
$A_\Theta$ be the corresponding self-adjoint extension via \eqref{atheta}-\eqref{atheta2}. 
Then the following assertions hold for all $\lambda\in\rho(A_0)$:
\begin{itemize}
\item [{\rm (i)}] $\lambda\in\sigma_{\rm p}(A_\Theta)$ if and only if
$0\in\sigma_{\rm p}(\Theta-M(\lambda))$, and in this case
\begin{equation*}
\ker(A_\Theta-\lambda)=\gamma(\lambda)\,\ker(\Theta-M(\lambda));
\end{equation*}

\item [{\rm (ii)}] $\lambda\in\sigma_{\rm d}(A_\Theta)$ if and only if
$0\in\sigma_{\rm d}(\Theta-M(\lambda))$;

\item [{\rm (iii)}] $\lambda\in\sigma_{\rm c}(A_\Theta)$ if and only if
$0\in\sigma_{\rm c}(\Theta-M(\lambda))$;

\item [{\rm (iv)}] $\lambda\in\sigma_{\rm ess}(A_\Theta)$ if and only if
$0\in\sigma_{\rm ess}(\Theta-M(\lambda))$;

\item [{\rm (v)}] $\lambda\in\rho(A_\Theta)$ if and only if
$0\in\rho(\Theta-M(\lambda))$, and in this case 
\begin{equation}\label{kreini}
(A_\Theta-\lambda)^{-1}=(A_0-\lambda)^{-1}
+\gamma(\lambda)\bigl(\Theta-M(\lambda)\bigr)^{-1}\gamma(\overline\lambda)^*.
\end{equation}
\end{itemize}
\end{theorem}

If the self-adjoint relation $\Theta$ is of the form \eqref{nanu33} and $\lambda\in\rho(A_0)$, then one has 
$\lambda\in\sigma_{\rm p}(A_\Theta)$ if and only if
$0\in\sigma_{\rm p}(\mathcal B-\mathcal A M(\lambda))$, and in this case
\begin{equation}\label{jaab}
\ker(A_\Theta-\lambda)=\gamma(\lambda)\,\ker(\mathcal B-\mathcal A M(\lambda)).
\end{equation}
Furthermore, by \cite[Theorem~A.3]{HRT25} one has for $\lambda\in\rho(A_0)$ that $\lambda\in\rho(A_\Theta)$ if and only if
$0\notin\sigma_{\rm p}(\mathcal B-\mathcal A M(\lambda))$ and $\ran \mathcal{A} \subset \ran (\mathcal B-\mathcal A M(\lambda))$, and in this case \eqref{kreini} can be written in the form
\begin{equation}\label{kreini2}
(A_\Theta-\lambda)^{-1}=(A_0-\lambda)^{-1}
+\gamma(\lambda)\bigl(\mathcal B-\mathcal A M(\lambda)\bigr)^{-1}\mathcal A \gamma(\overline\lambda)^*.
\end{equation}

\subsection{Semibounded self-adjoint extensions} \label{section_semibounded_abstract}
Let us now consider the special case that $S$ is a densely defined closed symmetric operator which is also semibounded from below. For our applications it is actually sufficient
to treat the case of a positive lower bound and we refer the reader for a more general treatment to \cite[Chapter 5]{BHS20}. Therefore,
from now on we assume that there is a constant $\kappa>0$ such that 
\begin{equation}\label{semi}
 (Sf,f)\geq \kappa \Vert f\Vert^2
\end{equation}
holds for all $f\in\dom S$, and we choose $\kappa>0$ with this property maximal, so that $\kappa$ is indeed the lower bound of $S$. As in this case the defect numbers of $S$ are automatically equal there exists a boundary triplet $\{\cG,\Upsilon_0,\Upsilon_1\}$
for the adjoint operator $S^*$. Let $\gamma$ and $M$ be the corresponding $\gamma$-field and Weyl function. If the self-adjoint extension 
$A_0=S^*\upharpoonright\ker\Upsilon_0$ is also semibounded, then the lower bound $\kappa_0$ of $A_0$ satisfies $\kappa_0\leq\kappa$ and the Weyl function regarded on the interval 
$(-\infty,\kappa_0)$ is a monotonously increasing (operator) function. Then the limit $M(-\infty)$ exists in the strong resolvent sense, 
$$
(M(-\infty)-\lambda)^{-1}=\lim_{\eta\rightarrow -\infty}(M(\eta)-\lambda)^{-1},\quad \lambda\in\dC\setminus\dR,
$$
and defines a self-adjoint relation $M(-\infty)$ in $\cG$, see \cite[Theorem 5.2.11]{BHS20}. Since $M(\eta)$ and $M(-\infty)$ are self-adjoint, convergence in the 
strong resolvent sense is equivalent to convergence in the strong graph sense, see, e.g. \cite[Corollary 1.9.6]{BHS20}.

It is often convenient and natural to work with 
the Friedrichs extension $S_F$ of $S$. The following proposition shows how $S_F$ can be identified as a self-adjoint extension of $S$ with the help of the Weyl function.

\begin{proposition}\label{sf}
Assume that $S$ is bounded from below as in \eqref{semi}, let  $\{\cG,\Upsilon_0,\Upsilon_1\}$ be a boundary triplet for $S^*$ such that 
$A_0=S^*\upharpoonright\ker\Upsilon_0$ is also semibounded, and let $M$ be the corresponding Weyl function. Then the Friedrichs extension $S_F$ corresponds 
to the self-adjoint relation $M(-\infty)$, that is, 
\begin{equation*}
 S_F f=S^* f,\quad \dom S_F=\left\{f\in\dom S^*: \begin{pmatrix} \Upsilon_0 f \\ \Upsilon_1 f\end{pmatrix}\in M(-\infty)\right\}
\end{equation*}
or, in other words, $S_F=A_{M(-\infty)}$. In particular, if 
\begin{equation*}
M(-\infty)=\left\{\begin{pmatrix} 0 \\ g'\end{pmatrix}:g'\in\cG\right\},
\end{equation*}
then $S_F=S^*\upharpoonright\ker\Upsilon_0=A_0$.
\end{proposition}

In the following we fix a boundary triplet $\{\cG,\Upsilon_0,\Upsilon_1\}$ for $S^*$ such that $S_F=S^*\upharpoonright\ker\Upsilon_0=A_0$. This is possible according to 
\cite[Corollary 5.5]{BHS20} and leads to a particularly convenient situation when considering semibounded self-adjoint extensions of $S$. 
In particular, one has $(-\infty,\kappa)\subset\rho(A_0)$ and hence the Weyl function is a monotonously increasing (operator) function on $(-\infty,\kappa)$, where $\kappa>0$ is as in~\eqref{semi}.
We first parametrize the Krein--von Neumann extension $S_K$ of $S$, which is the smallest nonnegative extension and is (in an abstract form) given by
\begin{equation*}
 S_K f=S^* f,\quad \dom S_K=\dom S\dot +\ker S^*;
\end{equation*}
for different descriptions of the Krein--von Neumann extension and the general case $\kappa\geq 0$ we refer to \cite[Chapter 5.4]{BHS20},
the classical contributions \cite{AloSim80,AN70,G83,Kr47,Kr47a}
and, e.g., the more recent works \cite{AT09,AGMST13,FGKLNS21,HMD04} for an introduction and further references on Krein--von Neumann extensions and applications.
In the framework of self-adjoint extensions parametrized with the help of the boundary triplet $\{\cG,\Upsilon_0,\Upsilon_1\}$ the Krein--von Neumann extension $S_K$ corresponds 
to the self-adjoint operator $M(0)\in\cL(\cG)$, where $M$ is the Weyl function, that is, 
\begin{equation}\label{sk}
\begin{split}
 S_K f &=S^* f,\\ \dom S_K&=\bigl\{f\in\dom S^*: M(0)\Upsilon_0 f=\Upsilon_1 f\bigr\}=\ker(\Upsilon_1- M(0)\Upsilon_0),
\end{split}
\end{equation}
or, in other words, $S_K=A_{M(0)}$; cf. \eqref{atheta}-\eqref{atheta2}. In particular, if $M(0)=0$, 
then $S_K=S^*\upharpoonright\ker\Upsilon_1=A_1$. 

\begin{example} \label{example_KvN}
Consider the Weyl function in \eqref{mmm} and \eqref{m0} corresponding to the boundary triplet from Example~\ref{superklar} for $-\Delta$. 
It is not difficult to see that 
\begin{equation*}
 M(\eta)\rightarrow \left\{\begin{pmatrix} 0 \\ c\end{pmatrix}:c\in\dC^2\right\}
\end{equation*}
in the strong graph sense, and hence in the strong resolvent sense, as $\eta\rightarrow-\infty$. Therefore, Proposition~\ref{sf} shows that the Friedrichs extension $S_F$ of the minimal operator $S$
is given by $A_0=S^*\upharpoonright\ker\Upsilon_0$, 
\begin{equation*}
 S_F f=-f'',\quad \dom S_F=\bigl\{f\in W^{2,2} (\alpha,\beta): f(\alpha)=f(\beta)=0\bigr\}.
\end{equation*}
Note that in this situation the lower bound $\kappa$ of $S$ in \eqref{semi} and the lower bound $\kappa_0=\pi^2/(\beta-\alpha)^2$ of $A_0$ coincide.
For the Krein--von Neumann extension $S_K$ we use \eqref{sk} and \eqref{m0} and conclude 
\begin{equation*}
\begin{split}
 S_K f=-f'',\quad \dom S_K&=\left\{f\in W^{2,2} (\alpha,\beta): M(0)\begin{pmatrix}f(\alpha) \\ f(\beta)\end{pmatrix}=\begin{pmatrix}f'(\alpha) \\ -f'(\beta)\end{pmatrix}\right\}\\
 &=\left\{f\in W^{2,2} (\alpha,\beta): \frac{f(\beta)-f(\alpha)}{\beta-\alpha}= f'(\alpha) = f'(\beta)\right\};
\end{split}
 \end{equation*}
cf. \cite{AloSim80} for a similar characterization.
\end{example}

Now we turn to a discussion of semibounded self-adjoint extensions and their ordering.

\begin{proposition}\label{auchnoch}
Let $\{\cG,\Upsilon_0,\Upsilon_1\}$ be a boundary triplet for $S^*$ with corresponding Weyl
function $M$ and assume that $S_F=S^*\upharpoonright\ker\Upsilon_0$. Let $A_\Theta$ be a self-adjoint extension of $S$
corresponding to the self-adjoint relation $\Theta$ in $\cG$ via \eqref{atheta}-\eqref{atheta2}, and assume that $x < \kappa$. Then $M(x)\in \cL(\cG)$
and the following equivalence holds:
\begin{equation} \label{condition_semibounded}
x \leq A_\Theta \quad \text{if and only if} \quad
M(x) \leq  \Theta.
\end{equation}
If $\Theta$ is given in the form~\eqref{nanu33}-\eqref{nanu}, then 
\begin{equation} \label{condition_semibounded1}
x \leq A_\Theta \quad \text{if and only if} \quad
\cA \cB^*=\cB \cA^* \geq\cA M(x)\cA^*.
\end{equation}
In particular, if $A_\Theta$ is semibounded in $\sH$, then $\Theta$ is semibounded in $\cG$.
Furthermore, if $\cG$ is finite-dimensional or $(S_F-\lambda)^{-1}$ is compact for some $\lambda\in\rho(S_F)$ and $\Theta$ 
is semibounded in $\cG$, then  $A_\Theta$ 
is semibounded in $\sH$.
\end{proposition}
\begin{proof}
  While for most of the claims we refer to \cite[Chapter~5]{BHS20}, we shall give an argument that~\eqref{condition_semibounded} and~\eqref{condition_semibounded1} are equivalent. Indeed, if the self-adjoint relation $\Theta$ is given by~\eqref{nanu33}, then it can also be represented as in~\eqref{nanu} and thus,
  \begin{equation*}
  \Theta-M(x)=\left\{\begin{pmatrix}\cA^* h \\ (\cB^* - M(x)\cA^*) h \end{pmatrix}:h\in\cG\right\}.
\end{equation*}
By~\eqref{condition_semibounded} one has that $A_\Theta \geq x$ if and only if the latter relation is nonnegative, that is,  $((\cB^* - M(x)\cA^*) h,\cA^* h)\geq 0$ for $h\in\cG$, which can also be formulated as 
$$  \cA \cB^*=\cB \cA^* \geq\cA M(x)\cA^*.$$
This concludes the proof.
\end{proof}

\begin{corollary}
Let $\{\cG,\Upsilon_0,\Upsilon_1\}$ be a boundary triplet for $S^*$ and assume that $S_F=S^*\upharpoonright\ker\Upsilon_0$ and that also $A_1=S^*\upharpoonright\ker\Upsilon_1$
is semibounded. Let $A_{\Theta_1}$ and $A_{\Theta_2}$ be semibounded self-adjoint extensions of $S$ corresponding to the
semibounded self-adjoint relations $\Theta_1$ and $\Theta_2$.
Then
\begin{equation*}
 A_{\Theta_1} \leq A_{\Theta_2} \quad \text{if and only if} \quad \Theta_1 \leq \Theta_2.
\end{equation*}
\end{corollary}

Finally we characterize all nonnegative self-adjoint realizations of $-\Delta$ in $L^2(\alpha,\beta)$ in the next example.

\begin{example} \label{example_ordering}
It follows from Proposition~\ref{auchnoch} and \eqref{m0} that a self-adjoint realization $A_\Theta$ of $-\Delta$ in $L^2(\alpha,\beta)$ is nonnegative if and only if 
\begin{equation}\label{ineq}
 M(0)=\frac{1}{\beta-\alpha}
\begin{pmatrix} 
- 1 & 1\\ 1 & - 1
\end{pmatrix}\leq\Theta.
\end{equation}
Observe that \eqref{ineq} is satisfied by a symmetric $2\times 2$-matrix $\Theta$ 
if and only if the symmetric $2\times 2$-matrix  $\Theta-M(0)$ is nonnegative or, equivalently, the eigenvalues of $\Theta-M(0)$ are nonnegative (note that
the eigenvalues of $M(0)$ are given by $\{-\frac{2}{\beta - \lambda},0\}$). 
In the case that $\Theta$ is a self-adjoint relation with a nontrivial multivalued part it is more convenient to use a representation of 
$\Theta = \Theta^*$ as in \eqref{nanu} with $\cA,\cB\in\dC^{2\times 2}$. Using~\eqref{condition_semibounded1} we get that in this case $A_\Theta \geq 0$ if and only if
$$  \cA \cB^*=\cB \cA^* \geq\cA M(0)\cA^*.$$
\end{example}

\section{Some properties of H\"ormander transmission spaces} \label{sec3}

In this section we first summarize some notations for distributions and Sobolev spaces, and recall the concept
H\"ormander transmission spaces afterwards. We also collect some useful properties of these spaces that will be needed in Section~\ref{section_fractional},
where boundary triplets for the fractional Laplacian are constructed. 
Throughout this section it is always assumed that $-\infty < \alpha < \beta < \infty$. We also remark that the following considerations remain valid for bounded Lipschitz domains $\Omega \subset \mathbb{R}^n$.

\subsection{Distributions and Sobolev spaces} \label{section_Sobolev_spaces}

Let $\mathcal{D}(\mathbb{R})$ and $\mathcal{S}(\mathbb{R})$ be the spaces of compactly supported and rapidly decaying $C^\infty$-functions on $\mathbb{R}$, respectively.
Their duals $\mathcal{D}'(\mathbb{R})$ and $\mathcal{S}'(\mathbb{R})$ are the spaces of distributions and tempered distributions, respectively. Similarly, 
for the bounded open interval $(\alpha, \beta) \subset \mathbb{R}$ the space of all $C^\infty$-functions with compact support in $(\alpha, \beta)$ is denoted by $\mathcal{D}(\alpha, \beta)$ and the corresponding space of distributions is $\mathcal{D}'(\alpha, \beta)$. 
The restriction of distributions on $\mathbb{R}$ to the interval $(\alpha, \beta)$ is denoted by $\mathfrak{r}$, so  
\begin{equation} \label{def_restriction_op}
   \mathfrak{r}: \mathcal{D}'(\mathbb{R}) \rightarrow \mathcal{D}'(\alpha, \beta), \qquad \mathfrak{r} f = f_{|(\alpha, \beta)} := f \upharpoonright \mathcal{D}(\alpha, \beta),
\end{equation}
where $\mathcal{D}(\alpha, \beta)$ is embedded in $\mathcal{D}(\mathbb{R})$ via extension by zero.
The set of all tempered distributions supported in $[\alpha, \beta]$ is $\dot{\mathcal{S}}'([\alpha, \beta])$, i.e.
\begin{equation*}
  \dot{\mathcal{S}}'([\alpha, \beta]) = \bigl\{ f \in \mathcal{S}'(\mathbb{R}): f = 0 \text{ on } \mathbb{R} \setminus [\alpha, \beta]\bigr\},
\end{equation*}
where, as usual, for a (tempered) distribution $f=0$ on $\mathbb{R} \setminus [\alpha, \beta]$ means $f(\varphi)=0$ for all $\varphi\in \mathcal{D}(\mathbb R\setminus [\alpha, \beta])$.
Next, the Sobolev space $H^s(\mathbb{R})$ of order $s \in \mathbb{R}$ on $\mathbb{R}$ is defined by
\begin{equation*}
  H^s(\mathbb{R}) := \left\{ f \in \mathcal{S}'(\mathbb{R}): \int_{\mathbb{R}} (1 + x^2)^s |\mathcal{F} f(x)|^2 dx < \infty \right\},
\end{equation*}
where $\mathcal{F}$ is the Fourier transform on $\mathcal{S}'(\mathbb{R})$.
For a bounded interval $(\alpha, \beta) \subset \mathbb{R}$ we set 
\begin{equation*}
  \dot{H}^s([\alpha, \beta]) := \big\{ f\in H^s(\mathbb{R}):f = 0 \text{ on } \mathbb{R} \setminus [\alpha, \beta] \big\} = H^s(\mathbb{R})\cap \dot{\mathcal{S}}'([\alpha, \beta])
\end{equation*}
and
\begin{equation*}
\overline{H}^{s}(\alpha, \beta):=\{\mathfrak{r} F: F \in H^{s}(\mathbb{R})\}.
\end{equation*}
Likewise, we write
\begin{equation*}
  \overline{C}^\infty([\alpha, \beta]) := \big\{ f_{|[\alpha, \beta]}: f \in C^\infty(\mathbb{R}) \big\}.
\end{equation*}
Note that $\dot{H}^s([\alpha, \beta]) \subset \mathcal{S}'(\mathbb{R})$, while $\overline{C}^\infty([\alpha, \beta]), \overline{H}^{s}(\alpha, \beta) \subset \mathcal{D}'(\alpha, \beta)$. 

While in general the spaces $\overline{H}^s(\alpha, \beta)$ and $\mathfrak{r} \dot{H}^s([\alpha, \beta])$ are different, for suitable $s$ they coincide:

\begin{lemma} \label{lemma_H_dot_H_bar}
  For $s \in ( -\frac{1}{2}, \frac{1}{2})$ one has $\overline{H}^{s}(\alpha, \beta) = \mathfrak{r} \dot{H}^{s}([\alpha, \beta])$.
\end{lemma}
\begin{proof}
  To see the claim for $s \in [0, \frac{1}{2})$, note that $\dot{H}^{s}([\alpha, \beta])$ coincides with the set denoted in  \cite{M00} by $\widetilde{H}^s(\alpha, \beta)$,  see also \cite[Theorem~3.29~(ii)]{M00}. By \cite[Theorem~3.40]{M00} one has $\mathfrak{r} \widetilde{H}^s(\alpha, \beta) =  \overline{H}^s(\alpha, \beta)$, which is the claim for $s \in [0, \frac{1}{2})$.
  
  To see the claim for $s \in (-\frac{1}{2}, 0 )$, note that by \cite[Theorem~3.30~(i)]{M00} the set $\overline{H}^s(\alpha, \beta)$ coincides with the dual space $(\mathfrak{r}\widetilde{H}^{-s}(\alpha, \beta))^*$ of $\mathfrak{r} \widetilde{H}^{-s}(\alpha, \beta) = \mathfrak{r} \dot{H}^{-s}([\alpha, \beta])$, which is, by the first part of the proof, equal to $\overline{H}^{-s}(\alpha, \beta)$. Using again \cite[Theorem~3.30]{M00} one gets that $\overline{H}^{-s}(\alpha, \beta)^* = \mathfrak{r} \widetilde{H}^{s}(\alpha, \beta) = \mathfrak{r} \dot{H}^{s}([\alpha, \beta])$. Summing up, this means that $\overline{H}^s(\alpha, \beta) = \mathfrak{r} \dot{H}^{s}([\alpha, \beta])$, i.e. the claim is also true for $s \in (-\frac{1}{2}, 0 )$.
\end{proof}

Finally, we construct the extension by zero $\mathfrak{e} f$ of $f \in \overline{H}^s(\alpha, \beta)$ with $s > -\frac{1}{2}$.
For $s\geq 0$ the situation is clear as $\overline{H}^s(\alpha, \beta)\subset L^2(\alpha,\beta)$ and $\mathfrak{e}$ acts on $f \in H^s(\alpha, \beta)$ as
\begin{equation*}
  \mathfrak{e} f = \begin{cases} f & \text{ in } (\alpha, \beta), \\ 0 & \text{ in } \mathbb{R} \setminus (\alpha, \beta). \end{cases}
\end{equation*}
In the case $s\in (-\frac{1}{2},0)$
we rely on Lemma~\ref{lemma_H_dot_H_bar} and choose $F\in \dot{H}^s([\alpha, \beta])$ such that $f = \mathfrak{r} F$. Observe that $F\in \dot{H}^s([\alpha, \beta])$
is unique. In fact, if $F_1,F_2\in \dot{H}^s([\alpha, \beta])$ such that $\mathfrak{r} F_1=\mathfrak{r} F_2$, then $(\mathfrak{r} F_1-\mathfrak{r}F_2)(\varphi)=0$ for 
all $\varphi\in \mathcal{D}(\alpha, \beta)$ and hence 
$(F_1-F_2)(\mathfrak{e}\varphi)=0$ for the $0$-extensions 
$\mathfrak{e}\varphi$.
Since $-s\in (0,\frac{1}{2})$ it follows from \cite[Theorem~3.30 and Theorem 3.40]{M00} that
the $0$-extensions of $\mathcal{D}(\alpha, \beta)$-functions are dense in the dual of 
$\dot{H}^s([\alpha, \beta])$, which implies $F_1-F_2=0$. Therefore, if $s\in (-\frac{1}{2},0)$, then the extension operator
\begin{equation} \label{def_extension_op}
  \mathfrak{e}: \overline{H}^{s}(\alpha, \beta) \rightarrow \dot{H}^s([\alpha, \beta]), \quad \mathfrak{e} f = F,
\end{equation}
is well defined.
By construction, if $s > -\frac{1}{2}$ we have 
\begin{equation}\label{guenstig}
\mathfrak{r} \mathfrak{e} f = f\quad\text{and}\quad \mathfrak{e} \mathfrak{r} F = F\quad\text{for all } f \in \overline{H}^s(\alpha, \beta),\,F \in \dot{H}^s([\alpha, \beta]).
\end{equation}

\subsection{The H\"ormander transmission spaces $H^{\mu(s)}([\alpha, \beta])$} \label{section_Hoermander_transmission_spaces}

In this subsection we recall the notion of H\"ormander transmission spaces and we provide some useful properties of the spaces, that will be needed to
study the restricted fractional Laplacian in an operator theoretic framework. These spaces were originally introduced in \cite{H65} and further studied and developed in \cite{AG23, G15, G18}, see also \cite{G14, G16, G18-1, G19}.
We follow the exposition introduced in the before mentioned papers by Grubb.

In order to define the H\"ormander transmission spaces, we will make for $\mu > -\frac{1}{2}$ use of the order reducing operator $\Lambda_+^{(\mu)}$ from \cite[Theorem~1.3]{G15}, which is an elliptic pseudodifferential operator on $\mathbb{R}$ and classical of order $\mu$, and thus
\begin{equation*} 
  \Lambda_+^{(\mu)}: H^{s}(\mathbb{R}) \rightarrow H^{s-\mu}(\mathbb{R})
\end{equation*}
is bounded for all $s \in \mathbb{R}$. Moreover, this operator $\Lambda_+^{(\mu)}$
preserves the support in $[\alpha, \beta]$ (i.e. it maps elements of $\dot{\mathcal{S}}'([\alpha, \beta])$ to $\dot{\mathcal{S}}'([\alpha, \beta])$) and defines a homeomorphism 
\begin{equation} \label{mapping_property_Lambda}
  \Lambda_+^{(\mu)}: \dot{H}^{s}([\alpha, \beta]) \rightarrow \dot{H}^{s-\mu}([\alpha, \beta])
\end{equation}
for all $s \in \mathbb{R}$, with inverses $(\Lambda_+^{(\mu)})^{-1}$ that also preserve support in $[\alpha, \beta]$; cf. \cite[Theorem~1.3]{G15} for details. 

The next definition follows \cite[Definition~1.8]{G15}, see also \cite[Proposition~1.7 and the comments after Definition~1.8]{G15} for the equality in~\eqref{eq:basic-transmi-def}. We remark that the H\"ormander transmission spaces $H^{\mu(s)}([\alpha, \beta])$ are defined in \cite{G15} for a larger set of the parameter $\mu$, but motivated by our application we will state the definition for the case $\mu > -\frac{1}{2}$.
Recall that $\mathfrak{e}$ and $\mathfrak{r}$ denote the extension and restriction operator, see~\eqref{def_extension_op} and~\eqref{def_restriction_op}, respectively.

\begin{definition}\label{hoermi}
Let $\mu>-\frac{1}{2}$ and $s \in \mathbb{R}$ such that $s-\mu>-\frac{1}{2}$. Then the \textit{$\mu$-transmission space} is defined by 
\begin{equation}\label{eq:basic-transmi-def}
\begin{split}
H^{\mu(s)}([\alpha, \beta]):&=\Lambda_+^{(-\mu)}\mathfrak{e} \overline{H}^{s-\mu}(\alpha, \beta) \\
&= \Big\{ f \in \dot{\mathcal{S}}'([\alpha, \beta]): \Lambda_+^{(\mu)} f \in \dot{H}^{-1/2 + \varepsilon}([\alpha, \beta]) \text{ for some } \varepsilon > 0 \\
&\qquad \qquad \qquad \qquad \qquad \qquad \quad\,\,\,\, \text{ and } \mathfrak{r} (\Lambda_+^{(\mu)} f) \in \overline{H}^{s-\mu}(\alpha, \beta) \Big\}.
\end{split}
\end{equation}
\end{definition}

Note that $H^{\mu(s)}([\alpha, \beta])$ is a Banach space when equipped with the norm
\begin{equation} \label{norm_transmission_space}
  \| f \|_{\mu(s)} := \big\| \mathfrak{r} \Lambda_+^{(\mu)} f \big\|_{\overline{H}^{s-\mu}(\alpha, \beta)};
\end{equation}
cf. \cite[Definition~1.8]{G15}. Moreover, for $s \geq 0$ the inclusion $H^{\mu(s)}([\alpha, \beta]) \subset L^2(\mathbb{R})$ holds.

In the following lemma we collect some known properties of $H^{\mu(s)}([\alpha, \beta])$ that will be useful for us. For the claims in (i),(iii)\&(iv), we provide the simple arguments for the proof. Here and in the following, we will make use of the distance function $t$ from the boundary that is defined by
\begin{equation}\label{eq:dist-t}
t(x):=\dist(x,\partial (\alpha, \beta))=\min\{x - \alpha, \beta -x\}, \qquad x \in (\alpha, \beta).
\end{equation}

\begin{lemma} \label{lemma_H_mu_s}
  Assume that $\mu>-\frac{1}{2}$ and $s -\mu >  -\frac{1}{2}$. Then the following assertions hold for the 
  H\"ormander transmission spaces $H^{\mu(s)}([\alpha, \beta])$ in~\eqref{eq:basic-transmi-def}:
  \begin{itemize}
    \item[(i)] If $s - \mu < \frac{1}{2}$, then $$H^{\mu(s)}([\alpha, \beta]) = \dot{H}^s([\alpha, \beta]),$$ and, in particular,  $\mathfrak{r} H^{\mu(0)}([\alpha, \beta]) = L^2(\alpha, \beta)$.
     \item[(ii)] If $s - \mu > \frac{1}{2}$ and $s - \mu - \frac{1}{2} \notin \mathbb{N}$, then 
    \begin{equation*}
      \dot{H}^{s}([\alpha, \beta]) \subset H^{\mu(s)}([\alpha, \beta]) \subset \dot{H}^{s}([\alpha, \beta]) + \mathfrak{e} t^\mu \overline{H}^{s-\mu}(\alpha, \beta).
    \end{equation*}
    \item[(iii)] If $s_1 < s_2$, then $H^{\mu(s_2)}([\alpha, \beta])$ is continuously embedded in $H^{\mu(s_1)}([\alpha, \beta])$.
    \item[(iv)] For all $s > 0$ the space $\mathfrak{r} H^{\mu(s)}([\alpha, \beta])$ is compactly embedded in $L^2(\alpha, \beta)$.
    \item[(v)] The set $t^\mu \overline{C}^\infty([\alpha, \beta])$ is dense in $\mathfrak{r}H^{\mu(s)}([\alpha, \beta])$.
  \end{itemize}
\end{lemma}
\begin{proof}
  (i) Recall that for $s - \mu \in (-\frac{1}{2}, \frac{1}{2})$ one has $\overline{H}^{s-\mu}(\alpha, \beta) = \mathfrak{r} \dot{H}^{s-\mu}([\alpha, \beta])$ by Lemma~\ref{lemma_H_dot_H_bar}. Therefore, it follows from Definition~\ref{hoermi} that
  $$ H^{\mu(s)}([\alpha, \beta])=\Lambda_+^{(-\mu)}\mathfrak{e} \overline{H}^{s-\mu}(\alpha, \beta) = \Lambda_+^{(-\mu)}\mathfrak{e}\mathfrak{r} \dot{H}^{s-\mu}([\alpha, \beta])$$
  and hence \eqref{guenstig} and~\eqref{mapping_property_Lambda} show $H^{\mu(s)}([\alpha, \beta]) = \dot{H}^s([\alpha, \beta])$. Since $\mathfrak{r} \dot{H}^0([\alpha, \beta]) = L^2(\alpha, \beta)$, this implies (for $s=0$) that $\mathfrak{r} H^{\mu(0)}([\alpha, \beta]) = L^2(\alpha, \beta)$.
   
  (ii) This follows from \cite[Theorem~5.4]{G15}, see also \cite[equation~(13)]{G20}.
  
  (iii) is an immediate consequence of~\eqref{norm_transmission_space} and the fact that $\overline{H}^{s_2-\mu}(\alpha, \beta)$ is boundedly embedded in $\overline{H}^{s_1-\mu}(\alpha, \beta)$, whenever $s_1 < s_2$.
  
  (iv) Choose $\widetilde{s} \in (0, \min \{ \frac{1}{2}, s\}  )$. Then, by~(iii) we have that $\mathfrak{r}H^{\mu(s)}([\alpha, \beta])$ is continuously embedded in $\mathfrak{r}H^{\mu(\widetilde{s})}([\alpha, \beta])$, which coincides, by~(i), with $\mathfrak{r} \dot{H}^{\widetilde{s}}([\alpha, \beta])$. Since $\widetilde{s} > 0$, the latter space is compactly embedded in $L^2(\alpha, \beta)$, which yields the claim in~(iii).
  
  
  (v) This claim follows from \cite[Proposition~4.1]{G15}.
\end{proof}

Item~(vi) in the previous lemma suggests that for $f \in H^{\mu(s)}([\alpha, \beta])$ the function $t^{-\mu} f$ admits Dirichlet traces, if $s - \mu > \frac{1}{2}$, and that $t^{-\mu} f$ admits Dirichlet and Neumann traces, whenever $s - \mu > \frac{3}{2}$. In the following lemma we see that this is indeed the case and we collect some properties of the trace maps. The claims about the existence of the trace maps can be deduced from \cite[equation~(3.12)]{G18-1}, the statements about their kernels are mentioned in \cite[equation~(2.12)]{G19}, see also \cite[Section~5]{G15} for details.

\begin{lemma} \label{lemma_traces}
  Let $\mu>-\frac{1}{2}$, $s -\mu >  -\frac{1}{2}$, and $H^{\mu(s)}([\alpha, \beta])$ be defined by~\eqref{eq:basic-transmi-def}. Then, the following is true:
  \begin{itemize}
    \item[(i)] If $s - \mu > \frac{1}{2}$, then the map 
    \begin{equation*}
      t^\mu \overline{C}^\infty([\alpha, \beta]) \ni f \mapsto \gamma_0^{(\mu)} f:= \begin{pmatrix} (t^{-\mu} f)(\alpha) \\ (t^{-\mu} f)(\beta) \end{pmatrix} 
    \end{equation*}
    has a continuous extension $\gamma_0^{(\mu)}: \mathfrak{r} H^{\mu(s)}([\alpha, \beta]) \rightarrow \mathbb{C}^2$. Moreover, 
    \begin{equation*}
      \mathfrak{r} H^{(\mu+1)(s)}([\alpha, \beta]) = \big\{ f \in \mathfrak{r} H^{\mu(s)}([\alpha, \beta]): \gamma_0^{(\mu)}(f) = 0 \big\}.
    \end{equation*}
    \item[(ii)] If $s - \mu > \frac{3}{2}$, then the map   
    \begin{equation*}
      t^\mu \overline{C}^\infty([\alpha, \beta]) \ni f \mapsto \gamma_1^{(\mu)} f:= \begin{pmatrix} (t^{-\mu} f)'(\alpha) \\ -(t^{-\mu} f)'(\beta) \end{pmatrix} 
    \end{equation*}
    has a continuous extension $\gamma_1^{(\mu)}: \mathfrak{r} H^{\mu(s)}([\alpha, \beta]) \rightarrow \mathbb{C}^2$. Moreover, 
    \begin{equation*}
      \mathfrak{r} H^{(\mu+2)(s)}([\alpha, \beta]) = \big\{ f \in \mathfrak{r} H^{\mu(s)}([\alpha, \beta]): \gamma_0^{(\mu)}(f) = \gamma_1^{(\mu)}(f) = 0 \big\}.
    \end{equation*}
  \end{itemize}
\end{lemma}

\section{Self-adjoint fractional Laplacians} \label{section_fractional}

Throughout this section, let $(\alpha, \beta) \subset \mathbb{R}$ be a bounded interval 
and $a \in (\frac{1}{2}, 1)$ be fixed. The \textit{restricted fractional Laplacian} is the operator 
\begin{equation*}
  \mathfrak{r} (-\Delta)^a \mathfrak{e},
\end{equation*}  
where the restriction operator $\mathfrak{r}$ is given by~\eqref{def_restriction_op}, $(-\Delta)^a$ is the fractional Laplacian on $\mathbb{R}$ defined by~\eqref{def_fract_Laplace_R}, and $\mathfrak{e}$ is the extension operator introduced in~\eqref{def_extension_op}. Hence,  $\mathfrak{r} (-\Delta)^a \mathfrak{e}$ acts on a generic $f \in L^2(\alpha, \beta)$ by first extending it by zero to $\mathfrak{e} f \in L^2(\mathbb{R})$, applying $(-\Delta)^a$ to it, and then restricting the distribution $(-\Delta)^a \mathfrak{e} f \in H^{-2a}(\mathbb{R})$ to $(\alpha, \beta)$. By construction, 
\begin{equation*} 
  \mathfrak{r} (-\Delta)^a \mathfrak{e} f(x) = c_{a} \lim_{\varepsilon \searrow 0} \int_{|x-y| > \varepsilon} \frac{(\mathfrak{e} f)(x)-(\mathfrak{e} f)(y)}{|x-y|^{1+2a}} d y, \quad x \in (\alpha, \beta),
\end{equation*}
if the integral converges, where
\begin{equation} \label{def_c_a}
  c_a := \frac{a 4^a}{\sqrt{\pi} \Gamma(1-a)} \Gamma\left( \frac{1+2a}{2} \right).
\end{equation} 
We first recall the definitions and some known results about the minimal, the maximal, and the 
Dirichlet realization of the restricted fractional Laplacian using suitable H\"ormander transmission spaces
as their natural domains; cf. Remark~\ref{maxrem}.
Afterwards, in Section~\ref{btsec} we introduce a boundary triplet for the restricted fractional Laplacian and show in the following subsections, how it can be used to deduce several properties of self-adjoint realizations of it. 

\subsection{The minimal, the maximal, and the Dirichlet realization of $(-\Delta)^a$ on $(\alpha, \beta)$} \label{section_minimal_maximal_operator}

In the following the H\"ormander transmission space $H^{(a-1)(2a)}([\alpha, \beta])$ will
play a particularly important role. Note that here $\mu=a-1$ and $s=2a$ in Definition~\ref{hoermi}
and hence $s-\mu=a + 1\in (\frac{3}{2}, 2)$. 
We will start with the \textit{maximal realization} $(-\Delta)_\max^{a}$ of the restricted fractional Laplacian
\begin{equation} \label{def_T_max}
  \begin{split}
  (-\Delta)_\max^{a} f &:= \mathfrak{r} (-\Delta)^a \mathfrak{e} f, \\ \dom (-\Delta)_\max^{a} &:= \mathfrak{r} H^{(a-1)(2a)}([\alpha, \beta]),
\end{split}
  \end{equation}
where $(-\Delta)^a$ is the fractional Laplacian defined on $L^2(\mathbb{R})$ and mapping to $H^{-2a}(\mathbb{R})$,
and $\mathfrak{e}$ and $\mathfrak{r}$ denote the extension and restriction operator defined by~\eqref{def_extension_op} and~\eqref{def_restriction_op}, respectively.
It follows from \cite[Theorem~7.1~$3^\circ$]{G15} that $(-\Delta)_\max^{a}$ is a well-defined operator in $L^2(\alpha, \beta)$, i.e. for all  $f \in \mathfrak{r} H^{(a-1)(2a)}([\alpha, \beta]) \subset L^2(\alpha, \beta)$ one has $(-\Delta)_\max^{a} f \in L^2(\alpha, \beta)$. The terminology \textit{maximal} will be explained later in Remark~\ref{maxrem}.

Next, we define the \textit{minimal realization} $(-\Delta)_\min^{a}$ of the restricted fractional Laplacian by
\begin{equation} \label{def_T_min}
\begin{split}
  (-\Delta)_\min^{a} f& := (-\Delta)_\max^{a} f,\\
  \dom (-\Delta)_\min^{a}& := \bigl\{ \dom (-\Delta)_\max^{a} : \gamma_0^{(a-1)}(f) = \gamma_1^{(a-1)}(f) = 0 \bigr\}.
\end{split}  
\end{equation}
Taking Lemma~\ref{lemma_traces} and Lemma~\ref{lemma_H_mu_s}~(i) into account (note that $2a-(a+1)\in (-\frac{1}{2},\frac{1}{2})$), one finds that 
\begin{equation} \label{dom_T_min}
  \dom (-\Delta)_\min^{a} = \mathfrak{r} H^{(a+1)(2a)}([\alpha, \beta]) = \mathfrak{r} \dot{H}^{2a}([\alpha, \beta]).
\end{equation}
In particular, $\mathfrak{e} \dom (-\Delta)_\min^{a} \subset H^{2a}(\mathbb{R})$ and thus, the application of $(-\Delta)^a$ in the definition of $(-\Delta)_\min^{a}$ results in an element in $L^2(\mathbb{R})$. Furthermore, as $\dot{H}^{2a}([\alpha, \beta])$ is a closed subspace of $H^{2a}(\mathbb{R})$, the operator $(-\Delta)_\min^{a}$ is closed and symmetric.
It will follow later in Theorem~\ref{theorem_bt} that $(-\Delta)_\max^{a} = ((-\Delta)_\min^{a})^*$.

An important role in our considerations is played by the \textit{Dirichlet realization} of $(-\Delta)_\max^{a}$, which is defined by
\begin{equation} \label{def_Dirichlet}
\begin{split}
  (-\Delta)^a_\textup{D} f&:= (-\Delta)_\max^{a} f, \\
  \dom (-\Delta)^a_\textup{D} &:= \{ f \in \dom (-\Delta)_\max^{a}: \gamma_0^{(a-1)}(f) = 0 \}.
  \end{split}
\end{equation}
Note that Lemma~\ref{lemma_traces}~(i) implies 
$$
\dom (-\Delta)^a_\textup{D} = H^{a(2a)}([\alpha, \beta]).
$$
It will turn out later in Proposition~\ref{proposition_Dirichlet_Friedrichs} that $(-\Delta)^a_\textup{D}$ coincides with the Friedrichs extension of the minimal realization $(-\Delta)_\min^{a}$.
In the next lemma known basic properties of $(-\Delta)^a_\textup{D}$ are collected:

\begin{lemma} \label{lemma_Dirichlet_basic}
  The Dirichlet realization $(-\Delta)^a_\textup{D}$ is a positive self-adjoint operator in $L^2(\alpha, \beta)$ with purely discrete spectrum accumulating to $+\infty$.
\end{lemma}
\begin{proof}
  The fact that $(-\Delta)^a_\textup{D}$ is self-adjoint and positive is mentioned in \cite[Section~6.2]{G16}; it is clear that $(-\Delta)^a_\textup{D}$ 
  is unbounded and hence the spectrum must accumulate to $+\infty$.
  Moreover, as $\dom (-\Delta)^a_\textup{D} = \mathfrak{r}H^{a(2a)}([\alpha, \beta])$ is compactly embedded in $L^2(\alpha, \beta)$, see Lemma~\ref{lemma_H_mu_s}~(iii), it follows that the spectrum of $(-\Delta)^a_\textup{D}$ consists of discrete eigenvalues with finite multiplicities. 
\end{proof}

As mentioned above, we will show in Proposition~\ref{proposition_Dirichlet_Friedrichs} that $(-\Delta)^a_\textup{D}$ coincides with the Friedrichs extension of the minimal realization $(-\Delta)_\min^{a}$. This operator has been investigated in \cite{K12}, where, in particular, a two-term Weyl-type asymptotic law for the eigenvalues was found and according to
\cite[Proposition 3]{K12} all eigenvalues of $(-\Delta)^a_\textup{D}$ are simple.
However, for our further analysis, the results stated in Lemma~\ref{lemma_Dirichlet_basic} are sufficient.

%
%
%
%
%

\subsection{The boundary triplet and self-adjoint extensions}\label{btsec}

Let $(-\Delta)_\max^{a}$ be the maximal realization of $(-\Delta)^a$ on $[\alpha, \beta]$ defined in~\eqref{def_T_max} and let $t$ be the distance function introduced in~\eqref{eq:dist-t}. Define the boundary mappings
\begin{equation*}
  \Upsilon_0^{(a)}, \Upsilon_1^{(a)}: \dom (-\Delta)_\max^{a} = \mathfrak{r} H^{(a-1)(2a)}([\alpha, \beta]) \rightarrow \mathbb{C}^2
\end{equation*}
by
\begin{equation} \label{def_Ups0}
  \Upsilon_0^{(a)} f := \Gamma(a)\gamma_0^{(a-1)} f=\Gamma(a) \begin{pmatrix} (t^{1-a}f)(\alpha) \\ (t^{1-a} f)(\beta) \end{pmatrix}
\end{equation}  
and 
  \begin{equation} \label{def_Ups1}
  \Upsilon_1^{(a)} f := \Gamma(a+1)\gamma_1^{(a-1)} f =\Gamma(a+1) \begin{pmatrix} (t^{1-a}f)'(\alpha) \\ -(t^{1-a} f)'(\beta) \end{pmatrix},
\end{equation}
where $\Gamma(a)$ and $\Gamma(a+1)$ are evaluations of Euler's gamma function.
Observe that $\Upsilon_0^{(a)}, \Upsilon_1^{(a)}$ are well-defined by Lemma~\ref{lemma_traces}. 
In the next theorem, we verify that $\{ \mathbb{C}^2, \Upsilon_0^{(a)}, \Upsilon_1^{(a)} \}$ is a boundary triplet for $(-\Delta)_\max^{a}$. Recall that $(-\Delta)_\min^{a}$ and $(-\Delta)^a_\textup{D}$ are the minimal and the Dirichlet realization of $(-\Delta)^a$ defined in~\eqref{def_T_min} and~\eqref{def_Dirichlet}, respectively.

\begin{theorem}\label{theorem_bt}
The minimal realization $(-\Delta)_\min^{a}$ of the restricted fractional Laplacian is a densely defined closed symmetric operator in $L^2(\alpha, \beta)$
and its adjoint is the maximal realization $(-\Delta)_\max^{a}$, that is, 
\begin{equation*}
 \bigl((-\Delta)_\min^{a}\bigr)^* = (-\Delta)_\max^{a}\quad\text{and}\quad  (-\Delta)_\min^{a} = \bigl((-\Delta)_\max^{a}\bigr)^*.
\end{equation*}
Furthermore, $\{ \mathbb{C}^2, \Upsilon_0^{(a)}, \Upsilon_1^{(a)} \}$ 
 is a boundary triplet for $(-\Delta)_\max^{a}$ such that
  \begin{equation} \label{equation_A_0}
    (-\Delta)_\max^{a} \upharpoonright \ker \Upsilon_0^{(a)} = (-\Delta)^a_{\textup{D}}.
  \end{equation}
\end{theorem}

\begin{remark}\label{maxrem}  
(i) The fact that $((-\Delta)_\min^{a})^* = (-\Delta)_\max^{a}$ justifies the terminology \textit{maximal realization of the fractional Laplacian in $L^2(\alpha, \beta)$} for $(-\Delta)_\max^a$. To see this, we note first that $(-\Delta)^a$ is bounded from $H^{2a}(\mathbb{R})$ to $L^2(\mathbb{R})$ and hence, by self-adjointness and duality $(-\Delta)^a$ has an extension that is bounded from $L^2(\mathbb{R})$ to $H^{-2a}(\mathbb{R})$. Thus, for any $f \in L^2(\alpha, \beta)$ one has $\mathfrak{r} (-\Delta)^a \mathfrak{e} f \in \overline{H}^{-2a}(\alpha, \beta)$ and for all $g$ in the dual space 
$$(\overline{H}^{-2a}(\alpha, \beta))^* = \mathfrak{r} \dot{H}^{2a}([\alpha, \beta])= \dom (-\Delta)^a_\min$$ 
(see, e.g., \cite[Theorem~3.30~(i)]{M00} for the first identity above) one has 
\begin{equation*}
  \bigl( \mathfrak{r} (-\Delta)^a \mathfrak{e} f, g \bigr)_{\overline{H}^{-2a}(\alpha, \beta) \times (\overline{H}^{-2a}(\alpha, \beta))^*} = \bigl( f, \mathfrak{r} (-\Delta)^a \mathfrak{e} g \bigr)_{L^2(\alpha, \beta)} = \bigl( f, (-\Delta)^a_\min g \bigr)_{L^2(\alpha, \beta)},
\end{equation*}
where 
the (sesquilinear) duality product in the pairing $\overline{H}^{-2a}(\alpha, \beta) \times (\overline{H}^{-2a}(\alpha, \beta))^*$ was used. 
From this it follows that $f \in L^2(\alpha, \beta)$ satisfies $\mathfrak{r} (-\Delta)^a \mathfrak{e} f \in L^2(\alpha, \beta)$ if and only if 
$f\in \dom ((-\Delta)_\min^{a})^* = \dom (-\Delta)_\max^{a}$ and $$\mathfrak{r} (-\Delta)^a \mathfrak{e} f = ((-\Delta)^a_\min)^* f = (-\Delta)_\max^a f.$$

%

(ii) We also note that by~\eqref{def_T_min} 
  \begin{equation*} 
    (-\Delta)_\min^{a}=(-\Delta)_\max^{a} \upharpoonright \big( \ker \Upsilon_0^{(a)} \cap \ker \Upsilon_1^{(a)} \big).
  \end{equation*}
  This fact will be seen independently also in the proof of Theorem~\ref{theorem_bt} in~\eqref{equation_S_restriction} below.
  
  (iii) For $a = 1$ the definition of $\Upsilon_0^{(a)}$ and $\Upsilon_1^{(a)}$ coincides formally with the one in Example~\ref{superklar}, and hence the construction is in accordance with the one for the classical Laplacian; note also that $\mathfrak{r} H^{0(2)}([\alpha, \beta])=\overline{H}^{2}(\alpha, \beta)=W^{2,2}(\alpha,\beta)$.
\end{remark}

\begin{proof}[Proof of Theorem~\ref{theorem_bt}]
  We apply Theorem~\ref{inv} to show the claims; in the following the maximal realization $(-\Delta)_\max^{a}$ plays the role of the operator $T$ in Theorem~\ref{inv}. 
  If $\gamma_0$ denotes the Dirichlet trace and $\gamma_1 f = \gamma_0 (\partial_\tau  f )$ with $\tau$ being the normal coordinate, 
  then by \cite[Theorem 4.4]{G18-1} we have for all $f, g \in H^{(a-1)(2a)}([\alpha, \beta]) = \dom (-\Delta)_\max^{a}$ 
  \begin{equation*}
    \begin{split}
      \big( (-\Delta)_\max^{a} f, g \big)_{L^2(\alpha, \beta)} &- \big( f, (-\Delta)_\max^{a} g \big)_{L^2(\alpha, \beta)} \\
      &= \Gamma(a+1) \Gamma(a) \big( \gamma_1 (t^{1-a} f), \gamma_0 (t^{1-a} g) \big)_{\mathbb{C}^2} \\
      &\qquad - \Gamma(a+1) \Gamma(a) \big( \gamma_0 (t^{1-a} f), \gamma_1 (t^{1-a} g) \big)_{\mathbb{C}^2}.
    \end{split}
  \end{equation*}
  The latter equation can be simplified to
  \begin{equation*}
    \begin{split}
      \big( (-\Delta&)_\max^{a} f, g \big)_{L^2(\alpha, \beta)} - \big( f, (-\Delta)_\max^{a} g \big)_{L^2(\alpha, \beta)} \\
      &= \Gamma(a+1) \Gamma(a) \big[(t^{1-a}f)'(\alpha) \overline{(t^{1-a}g)(\alpha)} - (t^{1-a}f)'(\beta) \overline{(t^{1-a}g)(\beta)} \\
      &\qquad \qquad \qquad \quad  - (t^{1-a}f)(\alpha) \overline{(t^{1-a}g)'(\alpha)} + (t^{1-a}f)(\beta) \overline{(t^{1-a}g)'(\beta)} \big] \\
      &=  \big( \Upsilon_1^{(a)} f, \Upsilon_0^{(a)} g \big)_{\mathbb{C}^2} - \big( \Upsilon_0^{(a)} f, \Upsilon_1^{(a)} g \big)_{\mathbb{C}^2}.
    \end{split}
  \end{equation*}
  Therefore, the abstract Green's identity in Theorem~\ref{inv}~(i) is valid.
  
  Next, for $(c_1,c_2,c_3,c_4) \in \mathbb{C}^4$ choose a polynomial $g$ such that 
  \begin{equation*}
    g(\alpha) = c_1, \quad g(\beta) = c_2, \quad g'(\alpha) = c_3, \quad g'(\beta) = -c_4.
  \end{equation*}
  Then the function
  \begin{equation*} 
    f := t^{a-1} g \in t^{a-1} \overline{C}^\infty([\alpha, \beta]) \subset H^{(a-1)(2a)}([\alpha, \beta]) = \dom (-\Delta)_\max^{a}
  \end{equation*}
  satisfies $\Upsilon_0^{(a)} f = (c_1, c_2)$ and $\Upsilon_1^{(a)} f = (c_3, c_4)$. Since $(c_1,c_2,c_3,c_4) \in \mathbb{C}^4$ was arbitrary, it follows that $(\Upsilon_0^{(a)}, \Upsilon_1^{(a)}): \dom (-\Delta)_\max^{a} \rightarrow \mathbb{C}^4$ is surjective. Moreover, according to \eqref{def_T_min}-\eqref{dom_T_min} we have 
  \begin{equation*}
    \dom (-\Delta)^a_\max \upharpoonright \bigl( \ker \Upsilon_0^{(a)} \cap \ker \Upsilon_1^{(a)} \bigr)= \mathfrak{r} \dot{H}^{2a}([\alpha, \beta]),
  \end{equation*}
  which is dense in $L^2(\alpha, \beta)$. Therefore, also assumption~(ii) in Theorem~\ref{inv} is fulfilled.
  
  Finally, it is clear by the definition of the Dirichlet realization $(-\Delta)^a_{\textup{D}}$ in~\eqref{def_Dirichlet} that~\eqref{equation_A_0} is true. This shows, in particular, with Lemma~\ref{lemma_Dirichlet_basic} that $(-\Delta)^a_\max \upharpoonright \ker \Upsilon_0^{(a)}$ is self-adjoint, i.e. also the assumption in~(iii) of Theorem~\ref{inv} is satisfied. Therefore, it follows from Theorem~\ref{inv} that
\begin{equation} \label{equation_S_restriction}
  S:=(-\Delta)_\max^{a}\upharpoonright\bigl(\ker\Upsilon_0^{(a)}\cap\ker\Upsilon_1^{(a)}\bigr)
\end{equation}
is a densely defined closed
symmetric operator in $L^2(\alpha, \beta)$ such that $S^*=(-\Delta)_\max^{a}$ and
$\{\mathbb{C}^2,\Upsilon_0^{(a)},\Upsilon_1^{(a)}\}$ is a boundary triplet for $S^*$ with
$A_0=S^*\upharpoonright\ker\Upsilon_0^{(a)}$. Note that $S=(-\Delta)_\min^{a}$ by \eqref{def_T_min}.
\end{proof}

As a first application of Theorem~\ref{theorem_bt} we can characterize  extensions of $(-\Delta)_\min^{a}$, which are, for matrices  $\mathcal A,\mathcal B\in \mathbb{C}^{2 \times 2}$, of the form 
\begin{equation} \label{def_extension_fractional_Laplace}
  (-\Delta)_{\cA, \cB}^a := (-\Delta)_\max^{a} \upharpoonright \ker (\mathcal B\Upsilon_0^{(a)} -\mathcal A\Upsilon_1^{(a)}).
\end{equation}
The operator $(-\Delta)_{\cA, \cB}^a$ can be described more explicitly by
\begin{equation*} 
\begin{split}
(-\Delta)_{\cA, \cB}^a f &= (-\Delta)_\max^{a} f, \\
\dom (-\Delta)_{\cA, \cB}^a&=\bigl\{f\in\dom (-\Delta)_\max^{a}: \mathcal B\Upsilon_0^{(a)} f=\mathcal A\Upsilon_1^{(a)} f\bigr\}.
\end{split}
\end{equation*}
We will be particularly interested in self-adjoint realizations $(-\Delta)_{\cA, \cB}^a$. For this, we make the following assumption:

\begin{assumption} \label{assumption_parameter}
Let $\mathcal A,\mathcal B\in \mathbb{C}^{2 \times 2}$ be such that
\begin{equation*}
\cA\cB^*= \cB\cA^*
\quad \mbox{and} \quad (\cA \, | \, \cB) \in \mathbb{C}^{2 \times 4} 
\text{ has rank } 2.
\end{equation*}
\end{assumption}

The following statement is now an immediate consequence of Proposition~\ref{abcd} and the fact that $\dom (-\Delta)^a_\max = \mathfrak{r} H^{(a-1)(2a)}([\alpha, \beta])$ is compactly embedded in $L^2(\alpha, \beta)$; cf. Lemma~\ref{lemma_H_mu_s}~(iv).

\begin{proposition}\label{proposition_self_adjoint}
Let $\{\mathbb{C}^2,\Upsilon_0^{(a)},\Upsilon_1^{(a)}\}$ be the boundary triplet for $(-\Delta)_\max^{a}$ defined in~\eqref{def_Ups0}-\eqref{def_Ups1}. Then, an operator $A$ satisfying $(-\Delta)_\min^{a} \subset A \subset (-\Delta)_\max^{a}$ is self-adjoint in $L^2(\alpha, \beta)$ if and only if there exist matrices  $\mathcal A,\mathcal B\in \mathbb{C}^{2 \times 2}$ that satisfy Assumption~\ref{assumption_parameter} such that $A$ is of the form $A = (-\Delta)_{\cA, \cB}^a$ as in~\eqref{def_extension_fractional_Laplace}. Moreover, the spectrum of any self-adjoint extension of $(-\Delta)^a_\min$ is purely discrete.
\end{proposition}

\subsection{The $\gamma$-field and Weyl function}  \label{section_Weyl_function}

In this section we compute the $\gamma$-field and Weyl function for the boundary triplet $\{ \mathbb{C}^2, \Upsilon_0^{(a)}, \Upsilon_1^{(a)} \}$, see Section~\ref{section_Weyl_function_abstract} for the general definition of these objects. While the $\gamma$-field and Weyl function can only be computed explicitly for the spectral point $\lambda = 0$, see Proposition~\ref{proposition_M_0}, we can express the $\gamma$-field and Weyl function for all $\lambda \in \rho((-\Delta)_\textup{D}^a)$ in terms of the resolvent  $((-\Delta)_\textup{D}^a-\lambda)^{-1}$ in Corollary~\ref{corollary_gamma_Weyl}. Finally, in Proposition~\ref{proposition_Kreinformula_0} we use these operators to conclude a Krein type resolvent formula for a generic self-adjoint realization $(-\Delta)_{\cA, \cB}^a$ as in~\eqref{def_extension_fractional_Laplace}.

\begin{proposition} \label{proposition_M_0}
  Let $\{ \mathbb{C}^2, \Upsilon_0^{(a)}, \Upsilon_1^{(a)} \}$ be the boundary triplet for $(-\Delta)_\max^a$ defined in~\eqref{def_Ups0}-\eqref{def_Ups1} and define the functions $v_1, v_2 \in H^{(a-1)(2a)}([\alpha, \beta])$ by
  \begin{equation} \label{def_v_1_v_2}
  \begin{split}
    v_1(x) &= \left( 1 - \frac{4}{(\beta-\alpha)^2} \left( x - \frac{\alpha + \beta}{2} \right)^2 \right)^{a-1}, \\
    v_2(x) &= \frac{2}{\beta - \alpha} \left( x - \frac{\alpha + \beta}{2} \right) v_1(x),\qquad \quad x \in (\alpha, \beta). 
  \end{split}
  \end{equation}
  Then, for the $\gamma$-field $\gamma^{(a)}$ and the Weyl function $M^{(a)}$ associated with $\{ \mathbb{C}^2, \Upsilon_0^{(a)}, \Upsilon_1^{(a)} \}$ one has 
  \begin{equation*}
    \gamma^{(a)}(0)\begin{pmatrix} c_1 \\ c_2 \end{pmatrix} = \left( \frac{\beta - \alpha}{4} \right)^{a-1} \frac{1}{\Gamma(a)} \left( \frac{c_1 + c_2}{2} v_1 + \frac{c_2 - c_1}{2} v_2 \right)
  \end{equation*}
  and
  \begin{equation} \label{equation_M_0}
    M^{(a)}(0) = \frac{a}{\beta - \alpha} \begin{pmatrix} -a & 1 \\ 1 & -a \end{pmatrix}.
  \end{equation}
\end{proposition}

Observe that formally for $a=1$ in~\eqref{equation_M_0} one recovers exactly the formulas for $\gamma(0)$ and $M(0)$ for the case of the classical Laplacian, 
see~\eqref{g0}-\eqref{m0}. Moreover, as we will see in the proof, the functions $v_1, v_2$ in~\eqref{def_v_1_v_2} are a basis of $\ker (-\Delta)^a_\max$. In particular, constant and linear functions do not belong to $\ker (-\Delta)^a_\max$ as soon as $a\ne 1$, as they can not be represented as linear combinations of $v_1$ and $v_2$.

\begin{proof}[Proof of Proposition~\ref{proposition_M_0}]
   First, it is shown in \cite{D12} that the functions
   \begin{equation*}
     \widetilde{v}_1(y) := (1-y^2)^{a-1}_+, \quad \widetilde{v}_2(y) := 
y  \widetilde{v}_1(y), \quad y \in \mathbb{R},
   \end{equation*}
   satisfy the relation
   \begin{equation*}
     (-\Delta)^a \widetilde{v}_j(y) = 0, \quad  y \in (-1, 1), \, 
j \in \{ 1, 2\};
   \end{equation*}
   cf. \cite[Tables~1 and~2]{D12}. In other words, $\widetilde{v}_j(y)$ 
belongs to the kernel of the restricted fractional Laplacian on 
$(-1,1)$. Hence, we find that the functions $v_1, v_2$ 
in~\eqref{def_v_1_v_2}, which satisfy the relation
   \begin{equation*}
     v_j(x) = \widetilde{v}_j\left( \frac{2}{\beta - \alpha} \left( x - 
\frac{\alpha + \beta}{2} \right) \right), \quad x \in (\alpha, \beta), 
\, j \in \{ 1, 2\},
   \end{equation*}
   belong to $\ker (-\Delta)_\max^{a}$. Since $\{ \mathbb{C}^2, 
\Upsilon_0^{(a)}, \Upsilon_1^{(a)} \}$ is a boundary triplet for 
$(-\Delta)_\max^{a}$, it follows from \cite[Theorems~2.4.1 
and~2.5.1]{BHS20} that $\dim \ker (-\Delta)_\max^{a} = \dim \mathbb{C}^2 
= 2$ and therefore,
   \begin{equation*}
     \ker (-\Delta)_\max^{a} = \textup{span}\, \{ v_1, v_2 \}.
   \end{equation*}
   In particular, this implies that for all $(c_1, c_2) \in 
\mathbb{C}^2$ there exist constants $a_1 = a_1(c_1, c_2), a_2 = 
a_2(c_1,c_2) \in \mathbb{C}$ such that
   \begin{equation} \label{equation_gamma}
     \gamma^{(a)}(0) \begin{pmatrix} c_1 \\ c_2 \end{pmatrix} = a_1 v_1 
+ a_2 v_2.
   \end{equation}
   To proceed, we note that
   \begin{equation} \label{v_1_t}
     t(x)^{1-a} v_1(x) = \frac{4^{a-1}}{(\beta - \alpha)^{2a-2}} 
\begin{cases} (\beta - x)^{a-1}, &x \in (\alpha, \frac{\alpha + 
\beta}{2}), \\ (x - \alpha)^{a-1}, &x \in (\frac{\alpha + \beta}{2}, 
\beta),  \end{cases}
   \end{equation}
   and
   \begin{equation} \label{v_2_t}
     t(x)^{1-a} v_2(x) = \frac{4^{a-1}}{(\beta - \alpha)^{2a-2}} 
\frac{2}{\beta - \alpha} \left( x - \frac{\alpha + \beta}{2} \right) 
\begin{cases} (\beta - x)^{a-1}, &x \in (\alpha, \frac{\alpha + 
\beta}{2}), \\ (x - \alpha)^{a-1}, &x \in (\frac{\alpha + \beta}{2}, 
\beta).  \end{cases}
   \end{equation}
   Therefore, a direct calculation shows that
   \begin{equation}\label{upsi000}
     \Upsilon_0^{(a)} v_1 = \Gamma(a) \begin{pmatrix} (t^{1-a} 
v_1)(\alpha) \\ (t^{1-a} v_1)(\beta) \end{pmatrix} = \Gamma(a) 
\left(\frac{4}{\beta - \alpha}\right)^{a-1} \begin{pmatrix} 1 \\ 1 
\end{pmatrix}
   \end{equation}
   and
   \begin{equation*}
     \Upsilon_0^{(a)} v_2 = \Gamma(a) \begin{pmatrix} (t^{1-a} 
v_2)(\alpha) \\ (t^{1-a} v_2)(\beta) \end{pmatrix} = \Gamma(a) 
\left(\frac{4}{\beta - \alpha}\right)^{a-1} \begin{pmatrix} -1 \\ 1 
\end{pmatrix}.
   \end{equation*}
   Combining this with~\eqref{equation_gamma} and the definition of 
$\gamma^{(a)}(0)$ (see Definition~\ref{def_gamma_Weyl_abstract}), we get 
that
   \begin{equation*}
     \begin{pmatrix} c_1 \\ c_2 \end{pmatrix} = \Upsilon_0^{(a)} 
\gamma^{(a)}(0) \begin{pmatrix} c_1 \\ c_2 \end{pmatrix} = \Gamma(a) 
\left(\frac{4}{\beta - \alpha}\right)^{a-1} \left( a_1 \begin{pmatrix} 1 
\\ 1 \end{pmatrix} + a_2 \begin{pmatrix} -1 \\ 1 \end{pmatrix} \right).
   \end{equation*}
   Solving this linear system of equations yields
   \begin{equation*}
     a_1 = \left( \frac{\beta - \alpha}{4} \right)^{a-1} 
\frac{1}{\Gamma(a)} \frac{c_1 + c_2}{2} , \qquad a_2 = \left( 
\frac{\beta - \alpha}{4} \right)^{a-1} \frac{1}{\Gamma(a)} \frac{c_2 - 
c_1}{2},
   \end{equation*}
   which gives in~\eqref{equation_gamma} exactly the claimed formula for 
$\gamma^{(a)}(0)$.

   To get the claimed formula for $M^{(a)}(0)$, we note that a direct 
calculation using~\eqref{v_1_t} and~\eqref{v_2_t} yields
   \begin{equation*}
     \Upsilon_1^{(a)} v_1 = \Gamma(a+1) \begin{pmatrix} (t^{1-a} 
v_1)'(\alpha) \\ -(t^{1-a} v_1)'(\beta) \end{pmatrix} = \Gamma(a+1) 
\frac{4^{a-1} (1-a)}{(\beta - \alpha)^a} \begin{pmatrix} 1 \\ 1 
\end{pmatrix}
   \end{equation*}
   and
   \begin{equation*}
     \Upsilon_1^{(a)} v_2 = \Gamma(a+1) \begin{pmatrix} (t^{1-a} 
v_2)'(\alpha) \\ -(t^{1-a} v_2)'(\beta) \end{pmatrix} = \Gamma(a+1) 
\frac{4^{a-1} (1+a)}{(\beta - \alpha)^a} \begin{pmatrix} 1 \\ -1 
\end{pmatrix}.
   \end{equation*}
   Therefore, one gets for any $(c_1,c_2) \in \mathbb{C}^2$
   \begin{equation*}
     \begin{split}
       M^{(a)}(0) \begin{pmatrix} c_1 \\ c_2 \end{pmatrix} &= 
\Upsilon_1^{(a)} \gamma^{(a)} (0) \begin{pmatrix} c_1 \\ c_2 
\end{pmatrix}  \\
       &=  \left( \frac{\beta - \alpha}{4} \right)^{a-1} 
\frac{1}{\Gamma(a)} \left( \frac{c_1 + c_2}{2} \Upsilon_1^{(a)} v_1 + 
\frac{c_2 - c_1}{2} \Upsilon_1^{(a)} v_2 \right) \\
       &= \left( \frac{\beta - \alpha}{4} \right)^{a-1} 
\frac{4^{a-1}}{(\beta - \alpha)^a} 
\frac{\Gamma(a+1)}{\Gamma(a)}\begin{pmatrix} -a & 1 \\ 1 & -a 
\end{pmatrix} \begin{pmatrix} c_1 \\ c_2 \end{pmatrix},
     \end{split}
   \end{equation*}
   which shows the claimed formula for $M^{(a)}(0)$.
\end{proof}

In the following corollary we express the $\gamma$-field and Weyl function for the boundary triplet $\{ \mathbb{C}^2, \Upsilon_0^{(a)}, \Upsilon_1^{(a)} \}$ for general $\lambda \in \rho((-\Delta)^a_\textup{D})$ in terms of $\gamma^{(a)}(0)$, $M^{(a)}(0)$, and $((-\Delta)^a_\textup{D} - \lambda)^{-1}$. This result is an immediate consequence of Proposition~\ref{proposition_M_0},~\eqref{gamm2}, and~\eqref{Weyl2}.

\begin{corollary} \label{corollary_gamma_Weyl}
  Let $\{ \mathbb{C}^2, \Upsilon_0^{(a)}, \Upsilon_1^{(a)} \}$ be the boundary triplet for $(-\Delta)_\max^a$ defined in~\eqref{def_Ups0}-\eqref{def_Ups1}. Then, for all $\lambda \in \rho((-\Delta)^a_\textup{D})$ one has 
  \begin{equation*}
    \gamma^{(a)}(\lambda)=\bigl(1+\lambda \big((-\Delta)^a_\textup{D} - \lambda\big)^{-1}\bigr) \gamma^{(a)}(0)
  \end{equation*}
  and
  \begin{equation*}
    M^{(a)}(\lambda) = \frac{a}{\beta - \alpha} \begin{pmatrix} -a & 1 \\ 1 & -a \end{pmatrix}+\gamma^{(a)}(0)^*\bigl[\lambda+\lambda^2 \big((-\Delta)^a_\textup{D}-\lambda\big)^{-1}\bigr]\gamma^{(a)}(0).
  \end{equation*}
\end{corollary}

One can apply now directly Theorem~\ref{Kreinformula} and the considerations below Theorem~\ref{Kreinformula} (see \eqref{jaab} and \eqref{kreini2}) to obtain a Birman-Schwinger principle and a Krein type resolvent formula for all self-adjoint extensions of $(-\Delta)_\min^{a}$ from Proposition~\ref{proposition_self_adjoint}.

\begin{proposition}\label{proposition_Kreinformula_0}
Let $\{\mathbb{C}^2,\Upsilon_0^{(a)},\Upsilon_1^{(a)}\}$ be the boundary triplet for $(-\Delta)_\max^{a}$ defined in~\eqref{def_Ups0}-\eqref{def_Ups1}, suppose that $\mathcal A,\mathcal B\in \mathbb{C}^{2 \times 2}$ satisfy Assumption~\ref{assumption_parameter}, and let $(-\Delta)^a_{\cA, \cB}$ be given by~\eqref{def_extension_fractional_Laplace}.
Then the following assertions hold for all $\lambda\in\rho((-\Delta)^a_\textup{D})$:
\begin{itemize}
\item [{\rm (i)}] $\lambda\in\sigma_{\rm p}((-\Delta)^a_{\cA, \cB})$ if and only if
$0\in\sigma_{\rm p}(\cB- \cA M^{(a)}(\lambda) )$, and in this case
\begin{equation*}
\ker((-\Delta)^a_{\cA, \cB}-\lambda)=\gamma^{(a)}(\lambda)\ker(\cB- \cA M^{(a)}(\lambda)).
\end{equation*}
In particular, $0 \in\sigma_{\rm p}((-\Delta)^a_{\cA, \cB})$ if and only if
\begin{equation*}
\det \left(\cB-\frac{a}{\beta - \alpha} \cA \begin{pmatrix} -a & 1 \\ 1 & -a \end{pmatrix}  \right) = 0;
\end{equation*}

\item [{\rm (ii)}] $\lambda \in\rho((-\Delta)^a_{\cA, \cB})$ if and only if
$0\notin\sigma_\textup{p}(\mathcal B-\mathcal A M^{(a)}(\lambda))$ and in this case
\begin{equation*}
\big((-\Delta)^a_{\cA, \cB} - \lambda\big)^{-1}=\big((-\Delta)^a_{\textup{D}} - \lambda\big)^{-1}
+\gamma^{(a)}(\lambda)\bigl(\mathcal B-\mathcal A M^{(a)}(\lambda)\bigr)^{-1}\mathcal A \gamma^{(a)}(\overline{\lambda})^*.
\end{equation*}
\end{itemize}
\end{proposition}

An immediate consequence of the previous considerations is the fact that all self-adjoint extensions of $(-\Delta)_\min^{a}$ are semibounded from below.

\begin{proposition} \label{prop_semibounded}
  Let $\{\mathbb{C}^2,\Upsilon_0^{(a)},\Upsilon_1^{(a)}\}$ be the boundary triplet for $(-\Delta)_\max^{a}$ defined in~\eqref{def_Ups0}-\eqref{def_Ups1}, suppose that $\mathcal A,\mathcal B\in \mathbb{C}^{2 \times 2}$ satisfy Assumption~\ref{assumption_parameter}, and let $(-\Delta)^a_{\cA, \cB}$ be given by~\eqref{def_extension_fractional_Laplace}. 
  Then $\sigma((-\Delta)^a_{\cA, \cB})\cap 
  (-\infty,0]$ consists of at most two eigenvalues (taking multiplicities into account), and the positive spectrum (discrete eigenvalues) of $(-\Delta)^a_{\cA, \cB}$
  accumulates to $+\infty$. In particular, $(-\Delta)^a_{\cA, \cB}$ is semibounded from below.
\end{proposition}
\begin{proof}
  Recall that $(-\Delta)^a_\textup{D}$ is a self-adjoint positive operator by Lemma~\ref{lemma_Dirichlet_basic}. Denote the lower bound of $(-\Delta)^a_\textup{D}$ by $\kappa_0>0$. It follows from Proposition~\ref{proposition_Kreinformula_0} that the resolvent difference 
  \begin{equation*}
\big((-\Delta)^a_{\cA, \cB} - i \big)^{-1}-\big((-\Delta)^a_{\textup{D}} - i \big)^{-1} 
= \gamma^{(a)}(i)\bigl(\mathcal B-\mathcal A M^{(a)}(i)\bigr)^{-1}\mathcal A \gamma^{(a)}(-i)^*
\end{equation*}
has at most rank two. Therefore, by \cite[\S~9.3, Theorem~3]{BS87} the set $\sigma((-\Delta)^a_{\cA, \cB}) \cap (-\infty, \kappa_0)$ consists of at most two eigenvalues (taking multiplicities into account) and hence $(-\Delta)^a_{\cA, \cB}$ is semibounded from below.
As $(-\Delta)^a_{\cA, \cB}$ is an unbounded self-adjoint operator the spectrum must accumulate to $+\infty$. According to Proposition~\ref{proposition_self_adjoint} the spectrum is purely discrete.
\end{proof}

\subsection{Further facts about $(-\Delta)^a_{\rm D}$}

Using the boundary triplet $\{\mathbb{C}^2,\Upsilon_0^{(a)},\Upsilon_1^{(a)}\}$  from~\eqref{def_Ups0}-\eqref{def_Ups1} we can prove some further results about the restricted fractional Laplacian with Dirichlet boundary conditions. Our starting point is the 
observation that $(-\Delta)^a_\textup{D}$ is the Friedrichs extension of $(-\Delta)_\min^{a}$; see also  \cite{G14, G15}.
The following result is a direct application of Proposition~\ref{sf}.

\begin{proposition}\label{proposition_Dirichlet_Friedrichs}
  The Friedrichs extension of $(-\Delta)_\min^{a}$ is given by $(-\Delta)^a_\textup{D}$.
\end{proposition}
\begin{proof}
  By Proposition~\ref{sf} the Friedrichs extension $S_F$ of $(-\Delta)_\min^{a}$ is the restriction of $(-\Delta)_\max^{a}$ to the set 
  \begin{equation*}
    \dom S_F = \left\{f\in \dom (-\Delta)_\max^{a}: \begin{pmatrix} \Upsilon_0^{(a)} f \\ \Upsilon_1^{(a)} f\end{pmatrix}\in M^{(a)}(-\infty)\right\},
  \end{equation*}
  where the limit $M^{(a)}(-\infty)$ of the Weyl function $M^{(a)}$ is understood in the strong resolvent sense or, equivalently, in the strong graph sense.
  
  Since $(-\Delta)^a_\textup{D} = (-\Delta)_\max^{a} \upharpoonright \ker \Upsilon_0^{(a)}$, it suffices to prove that
  \begin{equation} \label{equation_M_infty}
    (M^{(a)}(-\infty))^{-1} = \lim_{\lambda \rightarrow -\infty} (M^{(a)}(\lambda))^{-1} = 0 \quad \text{ in } \mathbb{C}^{2 \times 2}.
  \end{equation}
  For that purpose, we consider the triplet $\{ \mathbb{C}^2, \widetilde{\Upsilon}_0^{(a)}, \widetilde{\Upsilon}_1^{(a)} \}$, where
  \begin{equation*}
    \widetilde{\Upsilon}_0^{(a)} = \Upsilon_1^{(a)} \quad \text{and} \quad \widetilde{\Upsilon}_1^{(a)} = -\Upsilon_0^{(a)}.
  \end{equation*}
  As $\{\mathbb{C}^2,\Upsilon_0^{(a)},\Upsilon_1^{(a)}\}$ is a boundary triplet for $(-\Delta)_\max^{a}$, also $\{ \mathbb{C}^2, \widetilde{\Upsilon}_0^{(a)}, \widetilde{\Upsilon}_1^{(a)} \}$ is a boundary triplet for $(-\Delta)_\max^{a}$, for which we have 
  \begin{equation*}
    \widetilde{A}_0 = (-\Delta)_\max^{a} \upharpoonright \ker \widetilde{\Upsilon}_0^{(a)} = (-\Delta)_\max^{a} \upharpoonright \ker \Upsilon_1^{(a)}
  \end{equation*}
  and the corresponding  Weyl function $\widetilde{M}^{(a)}$ is given by
  \begin{equation*}
    \widetilde{M}^{(a)}(\lambda) = -(M^{(a)}(\lambda))^{-1}, \quad  \lambda \in \rho((-\Delta)_\textup{D}^a) \cap \rho(\widetilde{A}_0).
  \end{equation*}
  From Proposition~\ref{prop_semibounded} it is clear that $\widetilde{A}_0$ is bounded from below and hence we can choose $\kappa_1<0$ such that 
  $\widetilde{A}_0\geq \kappa_1$. As $(-\Delta)_\textup{D}^a$ is positive by Lemma~\ref{lemma_Dirichlet_basic} it is clear that 
  $(-\infty, \kappa_1)$ is contained in $\rho((-\Delta)_\textup{D}^a) \cap \rho(\widetilde{A}_0)$.

  In the following consider $\lambda \in (-\infty, \kappa_1)$.  
  Since $\widetilde{A}_0$ is a restriction of $(-\Delta)_\max^{a}$ it is clear that 
  $\dom \widetilde{A}_0 \subset \dom (-\Delta)_\max^{a} = \mathfrak{r} H^{(a-1)(2a)}([\alpha, \beta])$ and with the help of the closed graph theorem one obtains that 
  \begin{equation*}
    (\widetilde{A}_0 - \lambda)^{-1}: L^2(\alpha, \beta) \rightarrow \mathfrak{r} H^{(a-1)(2a)}([\alpha, \beta])
  \end{equation*}
  is  bounded. It follows that the block diagonal operator 
  \begin{equation} \label{block_operator}
  \begin{split}
  &\begin{pmatrix}
   (\widetilde{A}_0 - \lambda)^{-1} & 0 \\ 0 & 0_{L^2(\dR\backslash (\alpha,\beta))}
   \end{pmatrix}: \\
   &\quad\begin{pmatrix} L^2(\alpha,\beta) \\ \{ 0_{L^2(\dR\backslash (\alpha,\beta))} \} \end{pmatrix} \cong \dot{H}^0([\alpha, \beta]) = H^{(a-1)(0)}([\alpha, \beta]) \rightarrow H^{(a-1)(2a)}([\alpha, \beta])
   \end{split}
  \end{equation}
  is also bounded. We will show at the end of the proof that the interpolation property
  \begin{equation} \label{interpolation_Hoermander_space}
    \big[ H^{(a-1)(2a)}([\alpha, \beta]), H^{(a-1)(0)}([\alpha, \beta]) \big]_\theta = H^{(a-1)(2 a \theta)}([\alpha, \beta]), \quad \theta \in (0,1),
  \end{equation}
  holds.
  Since the operator in~\eqref{block_operator} is a positive self-adjoint operator in 
  $\dot{H}^0([\alpha, \beta])$ it 
  follows with an interpolation argument, see \cite[Section~2.1 in Chapter~I]{LM72}, and~\eqref{interpolation_Hoermander_space} that for any $\theta \in (0, 1)$ also
    \begin{equation*}
  \begin{pmatrix}
   (\widetilde{A}_0 - \lambda)^{-\theta} & 0 \\ 0 & 0_{L^2(\dR\backslash (\alpha,\beta))}
   \end{pmatrix}:\begin{pmatrix} L^2(\alpha,\beta) \\\{ 0_{L^2(\dR\backslash (\alpha,\beta))} \} \end{pmatrix} \cong \dot{H}^0([\alpha, \beta]) \rightarrow H^{(a-1)(2a\theta)}([\alpha, \beta])
  \end{equation*}
  is bounded. This implies that
  \begin{equation*}
    (\widetilde{A}_0 - \lambda)^{-\theta}: L^2(\alpha, \beta) \rightarrow \mathfrak{r} H^{(a-1)(2 a \theta)}([\alpha, \beta])
  \end{equation*}
  is bounded. In particular, for $\theta \in (\frac{1}{2}(1-\frac{1}{2a}), \frac{1}{2})$ one has $2a\theta\in (a-\frac{1}{2}, a)$ and thus 
  $2a\theta-(a-1)\in (\frac{1}{2},1)$, so that
  Lemma~\ref{lemma_traces}~(i) can be applied. It follows that the mapping
  \begin{equation*}
    \widetilde{\Upsilon}_1^{(a)} (\widetilde{A}_0 - \lambda)^{-\theta} = -\Upsilon_0^{(a)} (\widetilde{A}_0 - \lambda)^{-\theta}: L^2(\alpha, \beta) \rightarrow \mathbb{C}^2
  \end{equation*}
  is bounded. Therefore, we can apply \cite[Theorem~6.1]{BLLR18}, according to which there exists a constant $C$ that is independent of $\lambda$ such that
  \begin{equation*}
    \big\| (M^{(a)}(\lambda))^{-1} \big\| = \big\| \widetilde{M}^{(a)}(\lambda) \big\| \leq \frac{C}{(\kappa_1 - \lambda)^{1 - 2 \theta}}, \quad \lambda \in (-\infty, \kappa_1),
  \end{equation*}
  holds. From this~\eqref{equation_M_infty} and thus also the claim of this proposition follows.

  It remains to prove~\eqref{interpolation_Hoermander_space}, which will be done with the help of \cite[Theorem~14.3 in Chapter~I]{LM72}. In the following we explain how the spaces in \cite[Theorem~14.3 in Chapter~I]{LM72} can be chosen and why the assumptions there are satisfied. First, set $X = Y = \dot{H}^0([\alpha, \beta]) = H^{(a-1)(0)}([\alpha, \beta])$, $\Phi = \dot{\mathcal{S}}'([\alpha, \beta])$, $\Psi = \mathfrak{r} \dot{\mathcal{S}}'([\alpha, \beta])$, and $\partial = \mathfrak{r} \Lambda_+^{(a-1)}$. Since $\Lambda_+^{(a-1)}$ is a classical pseudodifferential operator that preserves the support in $[\alpha, \beta]$, $\partial$ is well defined and continuous from $\Phi$ to $\Psi$. Next, set $\mathscr{X} = \overline{H}^{a+1}(\alpha, \beta)$, $\mathscr{Y} = \overline{H}^{1-a}(\alpha, \beta)$. Then, 
  \begin{equation}\label{notation_lions_magenes}
    (X)_{\partial, \mathscr{X}} := \bigl\{ f \in \dot{H}^0([\alpha, \beta]): \mathfrak{r} \Lambda_+^{(a-1)} f \in \overline{H}^{a+1}(\alpha, \beta) \bigr\} = H^{(a-1)(2a)}([\alpha, \beta]);
  \end{equation}
  the inclusion $\subset$ in the latter equality is clear by the second line in~\eqref{eq:basic-transmi-def}, and the other inclusion $\supset$ also follows from the second line in~\eqref{eq:basic-transmi-def} using that $H^{(a-1)(2a)}([\alpha, \beta]) \subset L^2(\alpha, \beta)$, as $s = 2a \geq 0$. Likewise, one has that
  \begin{equation*}
    (Y)_{\partial, \mathscr{Y}} := \bigl\{ f \in \dot{H}^0([\alpha, \beta]): \mathfrak{r} \Lambda_+^{(a-1)} f \in \overline{H}^{a-1}(\alpha, \beta) \bigr\} = H^{(a-1)(0)}([\alpha, \beta]) = L^2(\alpha, \beta),
  \end{equation*}
  where the latter equality is true by Lemma~\ref{lemma_H_mu_s}~(i). To proceed, set
  $\tilde{\mathscr{X}} = \tilde{\mathscr{Y}} = \mathfrak{r} \dot{H}^{1-a}([\alpha, \beta])$. Then, with Lemma~\ref{lemma_H_dot_H_bar} we see that 
  \begin{equation*}
    \mathscr{X} = \overline{H}^{1+a}(\alpha, \beta) \subset \overline{H}^{1-a}(\alpha, \beta) =  \mathfrak{r} \dot{H}^{1-a}([\alpha, \beta]) = \tilde{\mathscr{X}}
  \end{equation*}
  and $\mathscr{Y} = \overline{H}^{1-a}(\alpha, \beta) = \mathfrak{r} \dot{H}^{1-a}([\alpha, \beta]) = \tilde{\mathscr{Y}}$, i.e. assumption~i) in \cite[(14.23) in Chapter~I]{LM72} is fulfilled. Moreover, due to the properties of $\Lambda_+^{(a-1)}$ in~\eqref{mapping_property_Lambda} one has that $\partial = \mathfrak{r} \Lambda_+^{(a-1)}$ is bounded from $X = Y = \dot{H}^0([\alpha, \beta])$ to $\tilde{\mathscr{X}} = \tilde{\mathscr{Y}} = \mathfrak{r} \dot{H}^{1-a}([\alpha, \beta])$, i.e. ii) in \cite[(14.23) in Chapter~I]{LM72} is satisfied. Eventually, set $\mathscr{G} = (\Lambda^{(a-1)}_+)^{-1} \mathfrak{e}$. Using again the mapping properties of $\Lambda_+^{(a-1)}$ in~\eqref{mapping_property_Lambda} and~\eqref{guenstig} one has that $\mathscr{G}$ is bounded from $\tilde{\mathscr{X}} = \tilde{\mathscr{Y}} = \mathfrak{r} \dot{H}^{1-a}([\alpha, \beta])$ to $X = Y = \dot{H}^0([\alpha, \beta])$ and that $\mathscr{G}$ is the right inverse of $\partial$. Hence, also iii) in \cite[(14.23) in Chapter~I]{LM72} is true (with $r \equiv 0$). Therefore, we conclude from \cite[Theorem~14.3 in Chapter~I]{LM72} that for $\theta \in (0,1)$
  \begin{equation*} 
    \begin{split}
      \big[ H^{(a-1)(2a)}([\alpha, \beta])&, H^{(a-1)(0)}([\alpha, \beta]) \big]_\theta \\
      &= \big([\dot{H}^0([\alpha, \beta]),\dot{H}^0([\alpha, \beta])]_\theta \big)_{\partial, [\overline{H}^{1+a}(\alpha, \beta), \overline{H}^{1-a}(\alpha, \beta)]_\theta} \\
      &=\big(\dot{H}^0([\alpha, \beta]) \big)_{\partial, \overline{H}^{2 a \theta + 1 - a}(\alpha, \beta)} \\
      &=H^{(a-1)(2 a \theta)}([\alpha, \beta]), 
    \end{split}
  \end{equation*}
  where for the last equality again a similar argument as in~\eqref{notation_lions_magenes} is used. Thus,~\eqref{interpolation_Hoermander_space} is true.
\end{proof}

Since $(-\Delta)_\textup{D}^a$ is the Friedrichs extension of $(-\Delta)_\min^{a}$, we can relate it to its associated quadratic form. Note that this quadratic form is used as starting point for the analysis of a Dirichlet realization of the fractional Laplacian in, e.g., \cite{BonSirVaz15, D12, FG16, K12, SV14}; we refer to the review paper \cite{F18} for further references.
Recall that $\mathfrak{e}$ and $\mathfrak{r}$ are the extension and restriction operator defined by~\eqref{def_extension_op} and~\eqref{def_restriction_op}, respectively.

\begin{corollary} \label{corollary_quadratic_form}
  The operator $(-\Delta)_\textup{D}^a$ is the unique self-adjoint operator in $L^2(\alpha, \beta)$ that is associated with the closed quadratic form
  \begin{equation} \label{def_form}
    \begin{split}
      q[f] = \frac{c_{a}}{2}\int_{\R}\int_{\R} \frac{|\mathfrak{e} f(x)-\mathfrak{e} f(y)|^2}{|x-y|^{1+2a}}d x d y, \quad \dom q = \mathfrak{r} \dot{H}^a([\alpha, \beta]),
    \end{split}
  \end{equation}
  where $c_a$ is given by~\eqref{def_c_a}.
\end{corollary}

\begin{proof}
  Recall that the Friedrichs extension of $(-\Delta)_\min^{a}$ is the unique self-adjoint operator associated to the closure of the quadratic form
  \begin{equation*}
  q_\min[f,g]=((-\Delta)_\min^{a} f, g),\quad \dom q_\min =\dom (-\Delta)_\min^{a}= \mathfrak{r} \dot{H}^{2a}([\alpha, \beta]);
  \end{equation*}
  thus we have to verify that the closure of $q_\min$ is the form $q$ in \eqref{def_form}.
  In fact, note first that for
  $f \in \dom (-\Delta)_\min^{a}$ a direct calculation involving~\eqref{def_fract_Laplace_R} yields 
  \begin{equation*}
    ((-\Delta)_\min^{a} f, f) = \frac{c_{a}}{2}\int_{\R}\int_{\R} \frac{|\mathfrak{e} f(x)-\mathfrak{e} f(y)|^2}{|x-y|^{1+2a}}d x d y,
  \end{equation*}
  where the right hand side coincides with the Sobolev-Slobodeckii semi-norm in $H^a(\mathbb{R})$.
  This shows
  \begin{equation*}
   q_\min[f]+ \Vert f\Vert_{L^2(\alpha,\beta)}^2 = \Vert f\Vert_{H^a(\alpha,\beta)}^2,\quad f\in\dom q_\min.
  \end{equation*}
  Since $\dot{H}^{a}([\alpha, \beta])$ is a closed subspace of $H^a(\mathbb{R})$ and 
  and since $\dot{H}^{2a}([\alpha, \beta])$ is dense in $\dot{H}^{a}([\alpha, \beta])$ (with respect to the norm $\Vert\cdot\Vert_{H^a(\mathbb R)}$)
  it follows that $\dom q_\min$ is dense in $\mathfrak{r}\dot{H}^{a}([\alpha, \beta])$ 
  (with respect to the norm $\Vert\cdot\Vert_{H^a(\alpha,\beta)}$). Therefore, the closure of the form $q_\min$ coincides with 
  the quadratic form $q$ in \eqref{def_form}.
\end{proof}

It is known that the operator associated with the form~\eqref{def_form} 
generates a sub-Markovian process (see, e.g., \cite[Section~2]{BBC03}), where the corresponding Dirichlet form is denoted by $(\mathcal C,\mathcal F^D)$), hence a strongly-continuous semigroup $(e^{-t(-\Delta)^a_{\mathrm{D}}})_{t\ge 0}$ on $L^2(\alpha,\beta)$ that is positivity-preserving and $L^\infty$-contractive. This means that 
\[
e^{-t(-\Delta)^a_{\mathrm{D}}}f(x)=\int_\alpha^\beta p_t(x,y)f(y)dy,\qquad\hbox{for all }t>0,\ f\in L^2(\alpha,\beta),\ x\in (\alpha,\beta),
\]
for an integral kernel $p_t$ such that $p_t(x,y)\in [0,1]$ for all $t>0$ and a.e.\ $x,y\in (\alpha,\beta)$.
In particular, we can apply now a result from \cite{M24} to obtain a pointwise upper bound for all eigenfunctions of $(-\Delta)^a_\textup{D}$.

\begin{theorem}
Each eigenpair $(\lambda^{\mathrm{D}}_n,\varphi_n)$ of $(-\Delta)^a_{\mathrm{D}}$ satisfies the pointwise estimate
\[
\frac{|\varphi_n(x)|}{\|\varphi_n\|_\infty}\le  \frac{\lambda^{\mathrm{D}}_n}{\Gamma(2a+1)} \left(\frac{\beta-\alpha}{2}\right)^a \left(1-\frac{4}{(\beta-\alpha)^2} \left( x-\frac{\alpha+\beta}{2} \right)^2\right)^{a},\qquad x\in (\alpha, \beta).
\]
\end{theorem}
\begin{proof}
  First, by \cite[Table~1]{D12} the function
  \begin{equation*}
    \widetilde{f}(y) := \frac{1}{\Gamma(2a+1)} (1-y^2)^{a}_+, \quad  y \in \dR,
  \end{equation*}
  satisfies the relation
  \begin{equation*}
    (-\Delta)^a \widetilde{f}(y) = 1, \quad  y \in (-1, 1).
  \end{equation*}
  Therefore, the function
  \begin{equation*}
    f(x) := \frac{1}{\Gamma(2a+1)} \left(\frac{\beta-\alpha}{2}\right)^a \left(1-\frac{4}{(\beta-\alpha)^2} \left( x-\frac{\alpha+\beta}{2} \right)^2\right)^{a}, \quad  x\in (\alpha, \beta),
  \end{equation*}
  satisfies 
  \begin{equation*}
    (-\Delta)^a f(x) = 1, \quad  x \in (\alpha, \beta).
  \end{equation*}
  With a similar calculation as in \eqref{v_1_t} and \eqref{upsi000} one sees that  $\Upsilon_0^{(a)} f=0$ and hence $f \in \dom (-\Delta)_\textup{D}^a$. 
Using the known fact that $-(-\Delta)^a_{\mathrm{D}}$ generates a positivity-preserving operator, we deduce the claim from \cite[Theorem~1.1]{M24}.
\end{proof}

\subsection{Semibounded self-adjoint extensions}

Since $(-\Delta)^a_\textup{D} = (-\Delta)_\max^{a} \upharpoonright \ker \Upsilon_0^{(a)}$ coincides with the Friedrichs extension of $(-\Delta)_\min^{a}$ 
by Proposition~\ref{proposition_Dirichlet_Friedrichs} the general principles from Section~\ref{section_semibounded_abstract} imply
further results about the ordering of self-adjoint extensions of $(-\Delta)_\min^{a}$. First, we describe the Krein--von Neumann extension of $(-\Delta)_\min^{a}$, i.e. the smallest nonnegative self-adjoint extension of $(-\Delta)_\min^{a}$; cf.~\eqref{sk}.
In the following result the boundary condition is explicit as $M^{(a)}(0)$ is known from Proposition~\ref{proposition_M_0}; 
again, by setting formally $a=1$, the result below reduces to the well-known formula for the Krein--von Neumann extension of the classical Laplacian in Example~\ref{example_KvN}.

\begin{proposition} \label{proposition_Krein_von_Neumann}
  The Krein-von Neumann extension of $(-\Delta)_\min^{a}$ is given by 
  \begin{equation*}
    \begin{split}
      (-\Delta)^a_\textup{K} f &= (-\Delta)_\max^{a} f, \\
      \dom (-\Delta)^a_\textup{K} &= \bigg\{ f \in H^{(a-1)(2a)}([\alpha, \beta]): \\
      &\quad \qquad \qquad  \frac{1}{\beta - \alpha} \begin{pmatrix} -a & 1 \\ 1 & -a \end{pmatrix} \begin{pmatrix} (t^{1-a}f)(\alpha) \\ (t^{1-a} f)(\beta) \end{pmatrix} = \begin{pmatrix} (t^{1-a}f)'(\alpha) \\ -(t^{1-a} f)'(\beta) \end{pmatrix} \bigg\}.
    \end{split}
  \end{equation*}
\end{proposition}

Similarly, Proposition~\ref{auchnoch} together with the explicit form of
$M^{(a)}(0)$ lead to a characterization of all nonnegative extensions of $(-\Delta)_\min^{a}$.

\begin{proposition}
Let $\{\mathbb{C}^2,\Upsilon_0^{(a)},\Upsilon_1^{(a)}\}$ be the boundary triplet for $(-\Delta)_\max^{a}$ defined in~\eqref{def_Ups0}-\eqref{def_Ups1}, suppose that $\mathcal A,\mathcal B\in \mathbb{C}^{2 \times 2}$ satisfy Assumption~\ref{assumption_parameter}, and let $(-\Delta)^a_{\cA, \cB}$ be given by~\eqref{def_extension_fractional_Laplace}.
Then the following equivalence holds:
\begin{equation*}
0 \leq (-\Delta)^a_{\cA, \cB} \quad \text{if and only if} \quad
 \cA \cB^*=\cB \cA^* \geq\frac{a}{\beta - \alpha} \cA \begin{pmatrix} -a & 1 \\ 1 & -a \end{pmatrix}  \cA^*.
\end{equation*}
\end{proposition}

Finally, we turn our attention to the Neumann realization of the restricted fractional Laplacian that is defined by
\begin{equation} \label{def_Neumann}
  \begin{split}
    (-\Delta)_\textup{N}^a f &:= (-\Delta)_\max^a f, \\
    \dom (-\Delta)^a_\textup{N} &:= \bigl\{ f \in H^{(a-1)(2a)}([\alpha, \beta]): (t^{1-a} f)'(\alpha) = (t^{1-a} f)'(\beta)  = 0 \bigr\}.
  \end{split}
\end{equation}
The regularity of solutions of equations involving the higher dimensional counterpart of $(-\Delta)_\textup{N}^a$ is studied in \cite{G14}. We emphasize that $(-\Delta)_\textup{N}^a$ is a different fractional Neumann realization as the \textit{spectral fractional Neumann Laplacian} investigated in \cite{G16}, the \textit{regional fractional Neumann Laplacian} studied in \cite{W15}, and the fractional Laplacian with \textit{nonlocal Neumann boundary conditions} introduced in \cite{DROV17}, and these operators have different properties as the ones found for $(-\Delta)_\textup{N}^a$ in Proposition~\ref{proposition_Neumann} below; we refer to \cite[Section 6.2 and 6.3]{G16} for an overview of different Neumann boundary conditions for the fractional Laplace operator.

Observe that with the boundary triplet in~\eqref{def_Ups0}-\eqref{def_Ups1} the operator \eqref{def_Neumann} can also be written in the form 
\begin{equation}\label{neumi}
(-\Delta)^a_\textup{N} = (-\Delta)^a_\max \upharpoonright \ker \Upsilon_1^{(a)},
\end{equation}
from which we conclude some interesting spectral properties of $(-\Delta)^a_\textup{N}$ in the next proposition. 
We point out that $0 \notin \sigma_\textup{p}((-\Delta)^a_\textup{N})$, but $(-\Delta)^a_\textup{N}$ has one negative eigenvalue. This is a significant difference to what is known for the Neumann realization of the classical Laplacian; cf. Remark~\ref{kannnichtschaden} below.  Recall that $\gamma^{(a)}$ and $M^{(a)}$ are the $\gamma$-field and Weyl function for the boundary triplet $\{ \mathbb{C}^2, \Upsilon_0^{(a)}, \Upsilon_1^{(a)} \}$, see Section~\ref{section_Weyl_function}.

\begin{proposition} \label{proposition_Neumann}
  The operator $(-\Delta)_\textup{N}^a$ is self-adjoint in $L^2(\alpha, \beta)$ and the following holds for all $\lambda\in\rho((-\Delta)^a_\textup{D})$.
  \begin{itemize}
\item [{\rm (i)}] $\sigma((-\Delta)^a_\textup{N})$ is purely discrete and 
\begin{equation} \label{BS_Neumann}
\lambda\in\sigma_{\rm p}((-\Delta)^a_{\cA, \cB}) \quad \text{if and only if} \quad 0\in\sigma_{\rm p}(M^{(a)}(\lambda) ).
\end{equation}
In particular, $0 \notin \sigma((-\Delta)^a_\textup{N})$ and $(-\Delta)_\textup{N}^a$ has exactly one simple negative eigenvalue.

\item [{\rm (ii)}] $\lambda \in\rho((-\Delta)^a_{\textup{N}})$ if and only if
$0\notin\sigma_\textup{p}(M^{(a)}(\lambda))$ and in this case
\begin{equation*}
\big((-\Delta)^a_{\textup{N}} - \lambda\big)^{-1}=\big((-\Delta)^a_{\textup{D}} - \lambda\big)^{-1}
-\gamma^{(a)}(\lambda)M^{(a)}(\lambda)^{-1}\gamma^{(a)}(\overline{\lambda})^*.
\end{equation*}
\end{itemize}
\end{proposition}
\begin{proof}
From \eqref{neumi} it is clear that $(-\Delta)^a_\textup{N}$ is of the form~\eqref{def_extension_fractional_Laplace} with $\cA = I$ and $\cB = 0$. Therefore, by Proposition~\ref{proposition_self_adjoint} the operator $(-\Delta)^a_\textup{N}$ is self-adjoint with purely discrete spectrum.
Moreover, the Birman-Schwinger principle in~\eqref{BS_Neumann} and the Krein-type resolvent formula in~(ii) follow from Proposition~\ref{proposition_Kreinformula_0}. 
Note that the eigenvalues of $M^{(a)}(0)$ are given by 
$$
\mu_1(0) = \frac{a}{\beta-\alpha} (-a - 1)\quad\text{and}\quad \mu_2(0) = \frac{a}{\beta-\alpha} (-a + 1).
$$
As $\mu_1(0)\not= 0$ and $\mu_2(0)\not= 0$ we conclude from \eqref{BS_Neumann} that $0 \notin \sigma((-\Delta)^a_\textup{N})$. 
  
  It remains to show that $(-\Delta)^a_\textup{N}$ has exactly one negative eigenvalue. Let $\kappa_0$ be the lower bound of $(-\Delta)^a_\textup{D} = (-\Delta)^a_\max \upharpoonright \ker \Upsilon_0^{(a)}$, recall that $\kappa_0>0$, and denote by $\mu_1(\lambda) \leq \mu_2(\lambda)$ the eigenvalues of $M^{(a)}(\lambda)$, $\lambda \in (-\infty, \kappa_0)$. By the analyticity and monotonicity of $(-\infty, \kappa_0) \ni  \lambda \mapsto M^{(a)}(\lambda)$ also the functions
  \begin{equation*}
    (-\infty, \kappa_0) \ni  \lambda \mapsto \mu_j(\lambda), \qquad j\in \{ 1, 2 \},
  \end{equation*}
  are continuous and non-decreasing. This and~\eqref{equation_M_infty} imply that
  \begin{equation*}
    \lim_{\lambda \rightarrow -\infty} \mu_j(\lambda) = -\infty, \qquad j \in \{ 1, 2 \}.
  \end{equation*}
  Since $\mu_1(0) < 0$ the function $\mu_1$ has no negative zero. On the other hand, as $\mu_2(0)  > 0$, the function $\mu_2$ has exactly one zero in $(-\infty,0)$, which by ~\eqref{BS_Neumann} is an eigenvalue of $(-\Delta)^a_\textup{N}$. Thus,  $(-\Delta)^a_\textup{N}$ has exactly one negative simple eigenvalue.
\end{proof}

\begin{remark}\label{kannnichtschaden}
As mentioned below Proposition~\ref{proposition_M_0} the $2\times 2$-matrix $M^{(a)}(0)$ coincides with $M(0)$ in \eqref{m0} in the limiting case $a=1$.
The eigenvalues of $M(0)$ are given by $0$ and $\frac{-2}{\beta-\alpha}$ and hence the proof of Proposition~\ref{proposition_Neumann} would imply that 
the usual Neumann Laplacian (that is, $A_1$ in \eqref{a1a1}) is nonnegative and has $0$ as a simple eigenvalue (which is in accordance with the 
well understood spectral properties of this operator).
\end{remark}

\noindent {\bf Acknowledgments.} 
We are indebted to Krzysztof Bogdan and Markus Kunze for fruitful discussions and helpful remarks.
D.M.\ is most grateful for a stimulating research stay at Graz University of Technology, where some parts of this paper were written in October 2024, and acknowledges financial support by the Deutsche Forschungsgemeinschaft (Grant 397230547).
This research was funded in part by the Austrian Science Fund (FWF) 10.55776/P 33568-N. For the purpose of open access, the author has applied a CC BY public copyright licence to any Author Accepted Manuscript version arising from this submission. 

%

\end{document}